\documentclass{article}
\usepackage{graphicx}
\usepackage{amsmath,mathrsfs,amsthm,amssymb}
\usepackage{enumerate,enumitem}
\usepackage[thicklines]{cancel}
\usepackage{authblk}
\usepackage[all]{xy}
\title{Cohomologies, non-abelian extensions and Wells sequences of $\lambda$-weighted Rota-Baxter Lie coalgebras}
\author[a]{Qinxiu Sun \thanks{Corresponding author: \\Sun Qinxiu (qxsun@126.com), Li Zhen (lz200002272022@163.com).}}
\author[a]{Zhen Li}
\affil[a]{Department of Mathematics, Zhejiang University of Science and Technology, Hangzhou, 310023.}
\date{}
\newtheorem{Def}{Definition}[section]
\newtheorem{Theo}{Theorem}[section]
\newtheorem{rem}{Remark}[section]
\newtheorem{Prop}{Proposition}[section]

\newtheorem{lemma}{Lemma}[section]
\newtheorem{example}{Example}[section]
\numberwithin{equation}{section}
\begin{document}
	\maketitle
\begin{abstract}
In this paper, we investigate cohomologies and non-abelian
extensions of $\lambda$-weighted Rota-Baxter Lie coalgebras. First, we consider Lie comodules and cohomologies of
$\lambda$-weighted Rota-Baxter Lie coalgebras. Next, we study non-abelian extensions of $\lambda$-weighted Rota-Baxter
Lie coalgebras and classify the non-abelian extensions in terms of non-abelian cohomology group. Furthermore, we explore
extensibility of a pair of automorphisms about a non-abelian extension of $\lambda$-weighted Rota-Baxter Lie coalgebras,
and derive the fundamental sequences of Wells in the context of $\lambda$-weighted Rota-Baxter Lie coalgebras. Finally,
we discuss the previous results in the case of abelian extensions of $\lambda$-weighted Rota-Baxter Lie coalgebras.

\end{abstract}

{\bf MR Subject Classification 2020}: 16T15, 17B62, 17B56, 17B10.


\section{Introduction}
Rota-Baxter operators on associative algebras arose initially in Baxter's study of the fluctuation theory in probability \cite{op1},
which can be viewed as an algebraic abstraction of the integral operators.
On the other hand, Rota-Baxter operators on Lie algebras were first considered
by Kuperscmidt in the study of classical $r$-matrices \cite{op12}.
Since then, Rota-Baxter operators have been investigated in connection with
many mathematical and physical branches, including combinatorics \cite{GG}, number theory \cite{GZ}, operads and
quantum field theory \cite{CK}.

Rota-Baxter operators with arbitrary weight (also called weighted Rota-Baxter operators) were considered
in \cite{wop1,wop2}. They are related with weighted infinitesimal bialgebras, weighted Yang-Baxter equations \cite{wop3}, combinatorics of
rooted forests \cite{wop5}, post-Lie algebras and modified Yang-Baxter
equations \cite{wop1}. Recently, Tang, Bai, Guo and Sheng \cite{TBGS} developed the deformation and
cohomology theory of $\mathcal{O}$-operators (also called relative Rota-Baxter operators) on Lie algebras,
with applications to Rota-Baxter Lie algebras in mind. Later, Das in \cite{wop7,co5} investigated the cohomologies of Rota-Baxter operators
 of arbitrary weights on associative algebras and Lie algebras. Wang and Zhou in \cite{wop8} explored cohomology and homotopy theories of Rota-Baxter algebras
with any weight. There are some other related work
concerning cohomologies of Rota-Baxter operators of any weight, see \cite{CLS,GQ} and there references.

The notion of a coalgebra is dual to the notion of an algebra. The theory of coassociative coalgebras
has been developed for a long time within the framework of the theory of Hopf algebras. Lie coalgebras were investigated earlier in \cite{co1} by
W. Michaelis. It is well known that the dual of a coassociative coalgebra is
 an associative algebra and the dual of a Lie coalgebra is a Lie algebra. Regarding applications, Lie coalgebras are not only closely related with Lie bialgebras and
 quantum groups, but also appeared in various fields such as homotopy theory \cite{SW}, noncommutative
 geometry \cite{GS}. Specifically, a bialgebra structure for Rota-Baxter
 Lie algebras was studied in \cite{RC3}. In a recent study, cohomologies of Lie coalgebras were examined by Du and Tan in \cite{co7}. But so far, cohomologies of $\lambda$-weighted Rota-Baxter
Lie coalgebras are still not given. This is our first
motivation for writing the present paper.

Extensions are useful mathematical objects to understand the
underlying structures. The non-abelian extension is a relatively
general one among various extensions (e.g. central extensions,
abelian extensions, non-abelian extensions etc.). Non-abelian
extensions were first developed by Eilenberg and Maclane \cite{nonC3},
which induce to the low dimensional non-abelian cohomology group.
Then numerous works have been devoted to non-abelian extensions of
various kinds of algebras, such as Lie (super)algebras, Leibniz algebras, Lie 2-algebras, Lie Yagamuti algebras, Rota-Baxter groups,
Rota-Baxter Lie algebras and Rota-Baxter Leibniz algebras, see \cite{nab6, CSZ, DR, F1, GKM, nab13, nab10,nab12,nonC2} and their references.
The abelian extensions of Lie coalgebras were explored
in \cite{aut1,co7} and characterized in terms of coderivations of Lie coalgebras.
But little is known about the non-abelian extension of Lie coalgebras particularly $\lambda$-weighted Rota-Baxter
Lie coalgebras. This is the second motivation for writing the present paper.

Another interesting study related to extensions of algebraic structures is given by the extensibility and inducibility of a pair
of automorphisms. When a pair of automorphisms is inducible? This problem was first
 considered by Wells \cite{W1} for abstract groups and further studied in \cite{aut5,aut6}.
Since then, several authors have studied this subject further, see \cite{BS,nab13,nab9,nonC2} and references therein.
The extensibility problem of a pair of derivations in abelian extensions were investigated in \cite{D1,TFS}.
Recently, the extensibility problem of a pair of derivations and automorphisms was extended to
the context of abelian extensions of Lie coalgebras \cite{aut1}. As byproducts, the Wells short exact sequences were obtained for
various kinds of algebras \cite{DR, aut1,GKM, nab13,nab9,aut5,nonC2}, which connected various automorphism groups and the non-abelian second cohomology.
Motivated by these results, we study extensibility of a pair of automorphisms in a non-abelian extension of $\lambda$-weighted Rota-Baxter
Lie coalgebras. This is another motivation for writing the present paper. We give a necessary and sufficient condition for
 a pair of automorphisms to be extensible. We also derive the analogue of the Wells short exact sequences in the context of
  non-abelian extensions of $\lambda$-weighted Rota-Baxter Lie coalgebras.

The paper is organized as follows. In Section 2, we introduce Lie comodules of $\lambda$-weighted Rota-Baxter
Lie coalgebras. In Section 3, we consider cohomologies of $\lambda$-weighted Rota-Baxter
Lie coalgebras. In Section 4, we
investigate non-abelian extensions and classify the non-abelian
extensions using non-abelian 2-cocycles. In Section 5,
we study the problem of when a pair of automorphisms is extensible. We address the  necessary and sufficient condition for
 a pair of automorphisms to be extensible.
 In Section 6, we build Wells short exact sequences in the context of non-abelian extensions of $\lambda$-weighted Rota-Baxter
Lie coalgebras. Finally, we discuss these results in the case of abelian extensions.

Throughout the paper, let $k$ be a field. Unless otherwise
specified, all vector spaces and algebras are finite dimensional
over $k$.

\section{Comodules of $\lambda$-weighted Rota-Baxter Lie coalgebras}
In this section, we introduce the Lie comodules of $\lambda$-weighted Rota-Baxter Lie coalgebras.
 We begin with recalling definitions of Lie coalgebras and $\lambda$-weighted Rota-Baxter Lie
 coalgebras \cite{RC3,co1}.
\begin{Def}
	\begin{enumerate}[label=$(\roman*)$]
		\item A Lie coalgebra is a vector space $C$ together with a linear map $\Delta_C:C\rightarrow C\otimes C$ satisfying
		\begin{equation}\Delta_C=-\tau\Delta_C,\label{1.1}\end{equation}
		\begin{equation}(I\otimes\Delta_C)\Delta_C-(\Delta_C\otimes I)\Delta_C+(I\otimes\tau)(\Delta_C\otimes I)\Delta_C=0.\label{LC}
		\end{equation}
		\item Let $(C,\Delta_C)$ be a Lie coalgebra. A linear operator $R_C:C\rightarrow C$ is called a $\lambda$-weighted $(\lambda\in k)$ Rota-Baxter operator if
		\begin{align}
			(R_C\otimes R_C)\Delta_C = (I\otimes R_C+R_C\otimes I+\lambda )\Delta_C R_C\label{R}.
		\end{align}
	Moreover, a Lie coalgebra $(C,\Delta_C)$ with a $\lambda$-weighted Rota-Baxter operator $R_C$ is called a $\lambda$-weighted Rota-Baxter Lie coalgebra (Rota-Baxter Lie coalgebra of weight $\lambda$).
	\end{enumerate}
\end{Def}
\begin{Def}
	Let $(C,\Delta_C,R_C)$ and $(C',\Delta_{C'},R_{C'})$ be two $\lambda$-weighted Rota-Baxter Lie coalgebras. A homomorphism from $(C,\Delta_C,R_C)$ to $(C',\Delta_{C'},R_{C'})$ is a Lie coalgebra homomorphism $\varphi:C \rightarrow C'$ satisfying $R_{C'}\varphi=\varphi R_C$.
\end{Def}
Let $(C,\Delta_C,R_C)$ be a $\lambda$-weighted Rota-Baxter Lie coalgebra.
 Denote the set of all automorphisms of $(C,\Delta_C,R_C)$ by $\mathrm{Aut}(C)$.
Obviously, $\mathrm{Aut}(C)$ is a group.
A right Lie comodule (corepresentation) of a Lie coalgebra $(C,\Delta_C)$ is a tuple $(M,\rho)$, where $M$ is a vector space and $\rho:M\rightarrow M\otimes C$ is a linear map satisfying \begin{align}
			(I\otimes \Delta_C)\rho-(\rho\otimes I)\rho+(I\otimes \tau)(\rho\otimes I)\rho=0\label{com}.
		\end{align}

\begin{Def}
	 Let $(C,\Delta_C,R_C)$ be a $\lambda$-weighted Rota-Baxter Lie coalgebra.
A right Lie comodule of $(C,\Delta_C,R_C)$ is a triple $(M,\rho,R_M)$, where $(M,\rho)$ is
a right Lie comodule of $(C,\Delta_C)$ and $R_M:M\rightarrow M$ is a linear map such that
		\begin{align}
			(R_M\otimes R_C)\rho=(R_M\otimes I+I\otimes R_C+\lambda )\rho R_M.\label{rR}
		\end{align}
\end{Def}
\begin{example}
	Any $\lambda$-weighted Rota-Baxter Lie coalgebra is a right Lie comodule of itself, which is called the adjoint corepresentation.
\end{example}
\begin{Prop}
	Let $(C,\Delta_C,R_C)$ be a $\lambda$-weighted Rota-Baxter Lie coalgebra. Then $(C^*,[\ ,\ ]_{C^*},R_{C}^*)$ is a $\lambda$-weighted Rota-Baxter Lie algebra.
\end{Prop}
\begin{proof}
	It is well known that $C^*$ is a Lie algebra with the multiplication given by
\begin{align*}
		[f,g]_{C^*}=(f\otimes g)\Delta_C,~\forall~ f, g\in C^*.
	\end{align*} We only need to check that $R_{C}^*$ is a $\lambda$-weighted Rota-Baxter operator.
Define $ R_{C}^* (f)=f R_C,~\forall~ f\in C^*$. In view of (\ref{R}), for all $f, g\in C^*$,
we have
	\begin{align*}
		&[R_{C}^*(f),R_{C}^*(g)]_{C^*}\\
		=&[fR_{C},gR_{C}]_{C^*}\\
		=&(fR_{C}\otimes gR_{C})\Delta_C\\
		=&(f\otimes g)(R_C\otimes I+I\otimes R_C+\lambda)\Delta_CR_C\\
		=&(fR_C\otimes g)\Delta_CR_C+(f\otimes gR_C)\Delta_CR_C+\lambda (f\otimes g)\Delta_CR_C\\
		=&[fR_C,g]_{C^*}R_C+[f,gR_C]_{C^*}R_C+\lambda[f, g]_{C^*}R_C\\
		=&[R_{C}^*(f),g]_{C^*}R_C+[f,R_{C}^*(g)]_{C^*}R_C+\lambda[f, g]_{C^*}R_C\\
		=&R_{C}^*[R_{C}^*(f),g]_{C^*}+R_{C}^*[f,R_{C}^*(g)]_{C^*}+\lambda R_{C}^*[f, g]_{C^*}\\
		=&R_{C}^*\big([R_{C}^*(f),g]_{C^*}+[f,R_{C}^*(g)]_{C^*}+\lambda [f, g]_{C^*}\big).
	\end{align*}
This finishes the proof.
\end{proof}
\begin{Prop}\label{new}
	Let $(C,\Delta_C,R_C)$ be a $\lambda$-weighted Rota-Baxter Lie coalgebra and $(M,\rho,R_M)$ a right Lie comodule of $(C,\Delta_C,R_C)$.
Define two linear maps $\tilde{\Delta}_C:C\rightarrow C\otimes C$ and $\tilde{\rho}:M\rightarrow M\otimes C$ respectively by
	\begin{align}
		&\tilde{\Delta}_C=(I\otimes R_C+R_C\otimes I+\lambda )\Delta_C,\label{new1}\\
		&\tilde{\rho}=(I\otimes R_C)\rho-\rho R_M.\label{new2}
	\end{align}
	Then
	\begin{enumerate}[label=$(\roman*)$]
		\item $(C,\tilde{\Delta}_C,R_C)$ is a $\lambda$-weighted Rota-Baxter Lie coalgebra.\label{part1}
		\item $(M,\tilde{\rho},R_M)$ is a right Lie comodule of $(C,\tilde{\Delta}_C,R_C)$.\label{part2}
	\end{enumerate}
\end{Prop}
	\begin{proof}
		\begin{enumerate}[label=$(\roman*)$]
			\item It is obviously that (\ref{1.1}) holds. We only need to check Eqs.~(\ref{LC}) and (\ref{R}) hold
for $(\tilde{\Delta}_C,R_C)$. In the light of Eqs.~(\ref{new1}) and (\ref{LC}),
			\begin{align*}
				&(I\otimes\tilde{\Delta}_C)\tilde{\Delta}_C\\
				=&(I\otimes(I\otimes R_C+R_C\otimes I+\lambda )\Delta_C)(I\otimes R_C+R_C\otimes I+\lambda )\Delta_C\\
				=&(I\otimes R_C\otimes R_C+R_C\otimes I\otimes R_C+R_C\otimes R_C\otimes R_C+\lambda R_C\otimes I\otimes I\\&+\lambda I\otimes I\otimes R_C+I\otimes R_C\otimes I+\lambda )(I\otimes\Delta_C)\Delta_C,\\
				&(\tilde{\Delta}_C\otimes I)\tilde{\Delta}_C\\
				=&((I\otimes R_C+R_C\otimes I+\lambda )\Delta_C\otimes I)(I\otimes R_C+R_C\otimes I+\lambda )\Delta_C\\
				=&(I\otimes R_C\otimes R_C+R_C\otimes I\otimes R_C+R_C\otimes R_C\otimes I+\lambda I\otimes I\otimes R_C\\&+\lambda R_C\otimes I\otimes I+\lambda I\otimes R_C\otimes I+\lambda )(\Delta_C\otimes I)\Delta_C,\\
				&(I\otimes\tau)(\tilde{\Delta}_C\otimes I)\tilde{\Delta}_C\\
				=&(I\otimes\tau)(I\otimes R_C\otimes R_C+R_C\otimes I\otimes R_C+R_C\otimes R_C\otimes I+\lambda I\otimes I\otimes R_C\\&+\lambda R_C\otimes I\otimes I+\lambda I\otimes R_C\otimes I+\lambda )(\Delta_C\otimes I)\Delta_C\\
				=&(I\otimes R_C\otimes R_C+R_C\otimes I\otimes R_C+R_C\otimes R_C\otimes I+\lambda I\otimes I\otimes R_C\\&+\lambda R_C\otimes I\otimes I+\lambda I\otimes R_C\otimes I+\lambda )(I\otimes\tau)(\Delta_C\otimes I)\Delta_C,
			\end{align*}
		which indicate that
			\begin{align*}
				&(I\otimes\tilde{\Delta}_C)\tilde{\Delta}_C-(\tilde{\Delta}_C\otimes I)\tilde{\Delta}_C+(I\otimes\tau)(\tilde{\Delta}_C\otimes I)\tilde{\Delta}_C\\
				=&(I\otimes R_C\otimes R_C+R_C\otimes I\otimes R_C+R_C\otimes R_C\otimes I+\lambda I\otimes I\otimes R_C\\&+\lambda R_C\otimes I\otimes I+\lambda I\otimes R_C\otimes I+\lambda )\\&\times((I\otimes\Delta_C)\Delta_C-(\Delta_C\otimes I)\Delta_C+(I\otimes\tau)(\Delta_C\otimes I)\Delta_C)\\&=0.
			\end{align*}
			By Eqs.~(\ref{R}) and (\ref{new1}), we can directly calculate
			\begin{align*}
				&(R_C\otimes R_C)\tilde{\Delta}_C\\
				=&(R_C\otimes R_C)(I\otimes R_C+R_C\otimes I+\lambda )\Delta_C\\
				=&(I\otimes R_C+R_C\otimes I+\lambda )(R_C\otimes R_C)\Delta_C\\
				=&(I\otimes R_C+R_C\otimes I+\lambda )(I\otimes R_C+R_C\otimes I+\lambda )\Delta_CR_C\\
				=&(I\otimes R_C+R_C\otimes I+\lambda )\tilde{\Delta}_CR_C.
			\end{align*}
			Thus, $(C,\tilde{\Delta}_C,R_C)$ is a $\lambda$-weighted Rota-Baxter Lie coalgebra.
			\item Take the same procedure as the proof of~\ref{part1}.
		\end{enumerate}
	\end{proof}
\section{Cohomologies of $\lambda$-weighted Rota-Baxter Lie coalgebras}
In this section, we are devoted to studying cohomologies of $\lambda$-weighted Rota-Baxter Lie coalgebras.

In the following, we denote
\[a\otimes^{(1)}(b_1\otimes b_2\otimes b_3)=a\otimes b_1\otimes b_2\otimes b_3,\]
\[b_1\otimes a\otimes b_2\otimes b_3=a\otimes^{(2)}(b_1\otimes b_2\otimes b_3),\]
\[a\otimes^{(k)}(b_1\otimes\cdots\otimes b_k\otimes\cdots\otimes b_{n-1})=b_1\otimes\cdots\otimes a\otimes b_k\otimes\cdots\otimes b_{n-1},\]
\[b_1\otimes\cdots\otimes b_{n-1}\otimes a=a\otimes^{(n)}(b_1\otimes\cdots\otimes b_{n-1}),\]
\[(a_1\otimes a_2)\otimes^{(2)}(b_1\otimes b_2\otimes b_3)=b_1\otimes (a_1\otimes a_2)\otimes b_2\otimes b_3,\]
 for all $a,a_1,a_2,b_i\in C~(i=1,2\cdots)$.

At first, we recall cohomologies of Lie coalgebras studied in \cite{co7}.

Let $(C,\Delta_C)$ be a Lie coalgebra. Suppose that $(M,\rho)$ is a right Lie comodule of $(C,\Delta_C)$. Denote the set of $n$-cochains ($n\geq 0$)
 by $C^{n}(M,C)$, where
	\[C^{n}(M,C)=\mathrm{Hom}(M,\wedge^n C),~n\ge 0,\]
and the coboundary operator $\partial^n:C^n(M,C)\rightarrow C^{n+1}(M,C)$ is given by, for any $h\in C^n(M,C)$,
\begin{align}
	&\partial^0(h)=(h\otimes I)\rho,\\
	&\partial^n(h)=\frac{1}{2}\sum_{k=1}^{n}(-1)^k\mathrm{Alt}(\Delta_C\otimes^{(k)}I^{\otimes (n-1)})h+(-1)^{n-1}\mathrm{Alt}(h\otimes I)\rho,\ n\ge1,\label{coh1}
\end{align}
where $\mathrm{Alt}:\otimes^{n}C\longrightarrow \otimes^{n}C$ is given by
\[\mathrm{Alt}(c_1\otimes\cdots\otimes c_n)=\frac{1}{n!}\sum_{\sigma\in\mathfrak{S}_n}\mathrm{sgn}(\sigma)\psi_\sigma(c_1\otimes\cdots\otimes c_n),~\forall~c_i\in C.\]
 Denote the set of all n-cocycles and n-coboundaries respectively by $\mathrm{Z}^n(M,C)$ and $\mathrm{B}^n(M,C)$. Define
  $\mathrm{H}^n(M,C)=\mathrm{Z}^n(M,C)/\mathrm{B}^n(M,C),$
which is called the n-cohomology group of $(C,\Delta_C)$ with coefficients in $(M,\rho)$.

Moreover, let $(C,\Delta_C,R_C)$ be a $\lambda$-weighted Rota-Baxter Lie coalgebra and $(M,\rho,R_M)$ a right Lie comodule of it.
Proposition~\ref{new} indicates that $(M,\tilde{\rho},R_M)$ is a right Lie comodule of the $\lambda$-weighted Rota-Baxter
 Lie coalgebra $(C,\tilde{\Delta}_C,R_C)$. Consider the cohomology of $(C,\tilde{\Delta}_C)$ with
 coefficients in $(M,\tilde{\rho})$. Denote the set of $n$-cochains by
\[\tilde{C}^n(M,C)=\mathrm{Hom}(M,\wedge^n C),\]
and a coboundary map $\tilde{\partial}^n:\tilde{C}^n(M,C)\rightarrow \tilde{C}^{n+1}(M,C)$ given by, for any $h\in \tilde{C}^n(M,C)$,
\begin{align}
	\tilde{\partial}^0(h)=&(h\otimes I)\tilde{\rho}=(h\otimes R_C)\rho-(h\otimes I)\rho R_M,\\
	\tilde{\partial}^n(h)=&\frac{1}{2}\sum_{k=1}^{n}(-1)^k\mathrm{Alt}((I\otimes R_C+R_C\otimes I+\lambda)\Delta_C\otimes^{(k)}I^{\otimes (n-1)})h\notag\\
	&+(-1)^{n-1}\mathrm{Alt}(h\otimes R_C)\rho-(-1)^{n-1}\mathrm{Alt}(h\otimes I)\rho R_M,\ n\ge 1.\label{coh2}
\end{align}
Then $\{\tilde{C}^*(M,C),\tilde{\partial}^*\}$ is a cochain complex. The corresponding n-cohomology group is
\begin{align*}
	\mathrm{\tilde{H}}^n(M,C)=\mathrm{\tilde{Z}}^n(M,C)/\mathrm{\tilde{B}}^n(M,C),
\end{align*}
where $\mathrm{\tilde{Z}}^n(M,C)=\mathrm{Ker}(\tilde{\partial}^{n}),~\mathrm{\tilde{B}}^n(M,C)=\mathrm{Im}(\tilde{\partial}^{n-1}).$

Denote
\[R_C^{(i)_n}=\underbrace{(I\otimes\cdots\otimes R_C\otimes\cdots\otimes R_C\otimes\cdots\otimes I)}_{\hbox{where $R_C$ appears i times}},~\hbox{we sum up over all possible variants}.\]
\begin{Prop}\label{alt}
 For all $ c,c_1,c_2,c_i\in C$, we have
	\begin{enumerate}[label=$(\roman*)$]
		\item $\mathrm{Alt}(c)=c,~\mathrm{Alt}(c_1\otimes c_2)=\frac{1}{2}(c_1\otimes c_2-c_2\otimes c_1),~\mathrm{Alt}(\Delta_C)=\Delta_C.$\label{alt1}
		\item $\mathrm{Alt}\big((R_C\otimes R_C)(c_1\otimes c_2)\big)=(R_C\otimes R_C)\mathrm{Alt}(c_1\otimes c_2)=\frac{1}{2}((R_C\otimes R_C)(c_1\otimes c_2)-(R_C\otimes R_C)(c_2\otimes c_1))$.\label{alt2}
		\item $\mathrm{Alt}((R_C\otimes I+I\otimes R_C)(c_1\otimes c_2))=\frac{1}{2}\big((R_C\otimes I)(c_1\otimes c_2)-(I\otimes R_C)(c_2\otimes c_1)+(I\otimes R_C)(c_1\otimes c_2)-(R_C\otimes I)(c_2\otimes c_1)\big)=(R_C\otimes I+I\otimes R_C)\mathrm{Alt}(c_1\otimes c_2)$.\label{alt3}
		\item $\mathrm{Alt}(R_C^{\otimes n}(c_1\otimes\cdots\otimes c_n))=R_C^{\otimes n}\mathrm{Alt}(c_1\otimes\cdots\otimes c_n)$.\label{alt4}
		\item $\mathrm{Alt}(R^{(i)_n}_C(c_1\otimes\cdots\otimes c_n))=R^{(i)_n}_C\mathrm{Alt}(c_1\otimes\cdots\otimes c_n)$.\label{alt5}
	\end{enumerate}
\end{Prop}
\begin{proof}
	Items~\ref{alt1}-\ref{alt4} can be obtained easily.
	
	\ref{alt5} For any $c_i\in C~(i=1,2,3\cdots)$,
	due to $\psi_\sigma(R^{(i)_n}_C)=R^{(i)_n}_C$,
	\begin{align*}
		&\mathrm{Alt}(R^{(i)_n}_C(c_1\otimes\cdots\otimes c_n))\\
		=&\frac{1}{n!}\sum_\sigma \mathrm{sgn}(\sigma)\psi_\sigma(R^{(i)_n}_C(c_1\otimes\cdots\otimes c_n))\\
		=&\frac{1}{n!}\sum_\sigma \mathrm{sgn}(\sigma)\psi_\sigma(R^{(i)_n}_C)\psi_\sigma(c_1\otimes\cdots\otimes c_n)\\
		=&\frac{1}{n!}\sum_\sigma \mathrm{sgn}(\sigma)R^{(i)_n}_C\psi_\sigma(c_1\otimes\cdots\otimes c_n)\\
		=&R^{(i)_n}_C\mathrm{Alt}(c_1\otimes\cdots\otimes c_n).
	\end{align*}
\end{proof}
In the following, we characterize the relationship between the two cochain complexes
$\{C^*(M,C),\partial^*\}$ and $\{\tilde{C}^*(M,C),\tilde{\partial}^*\}$.

\begin{Prop}
	The collection of maps $\{\delta^n:C^n(M,C)\rightarrow \tilde{C}^n(M,C)\}_{n\ge 0}$ defined by, for any $h\in C^n(M,C)$,
	\begin{align}
		&\delta^0(h)=h,\\
		&\delta^n(h)=R_C^{\otimes n}h-\sum^{n-1}_{i=0 }\lambda^{n-i-1} R_C^{(i)_n}hR_M,\label{del}
	\end{align}
	is a homomorphism of cochain complexes from $\{C^*(M,C),\partial^*\}$ to $\{\tilde{C}^*(M,C),\tilde{\partial}^*\}$, that is,
	\begin{align*}
		\delta^{n+1}\partial^n=\tilde{\partial}^n\delta^n.
	\end{align*}
\end{Prop}
\begin{proof}
	Using Eqs.~(\ref{coh1}), (\ref{coh2}) and (\ref{del}), for all $m^*\in M^*$,
	\begin{align*}
			&\tilde{\partial}^{0}\delta^0(m^*)(m)-\delta^1\partial^0(m^*)(m)\\
			=&\tilde{\partial}^{0}(m^*)(m)-\delta^1((m^*\otimes I)\rho)(m)\\
			=&(m^*\otimes R_C)\rho(m)-(m^*\otimes I)\rho R_M(m)-R_C(m^*\otimes I)\rho(m)+(m^*\otimes I)\rho R_M(m)\\
			=&m^*(m_0)R_C(m_1)-m^*(m_0)R_C(m_1)\\=&0,
	\end{align*}
and by (\ref{alt}), for any $f\in C^1(M,C)$, we have
\begin{align*}
	&\tilde{\partial}^{1}\delta^1(f)-\delta^2\partial^1(f)\\
	=&\tilde{\partial}^{1}(R_Cf-fR_M)-\delta^2(-\frac{1}{2}\mathrm{Alt}\Delta_Cf+\mathrm{Alt}(f\otimes I)\rho)\\
	=&-\frac{1}{2}\mathrm{Alt}(I\otimes R_C+R_C\otimes I+\lambda)\Delta_C(R_Cf-fR_M)+\mathrm{Alt}((R_Cf-fR_M)\otimes R_C)\rho\\&-\mathrm{Alt}((R_Cf-fR_M)\otimes I)\rho R_M-(R_C\otimes R_C)(-\frac{1}{2}\mathrm{Alt}\Delta_Cf+\mathrm{Alt}(f\otimes I)\rho)\\&+(R_C\otimes I+I\otimes R_C+\lambda)(-\frac{1}{2}\mathrm{Alt}\Delta_Cf+\mathrm{Alt}(f\otimes I)\rho)R_M\\
	=&-\frac{1}{2}\mathrm{Alt}(R_C\otimes R_C)\Delta_Cf+\frac{1}{2}\mathrm{Alt}(I\otimes R_C+R_C\otimes I+\lambda)\Delta_CfR_M+\mathrm{Alt}(R_Cf\otimes R_C)\rho\\
	&-\mathrm{Alt}(fR_M\otimes R_C)\rho-\mathrm{Alt}(R_Cf\otimes I)\rho R_M+\mathrm{Alt}(fR_M\otimes I)\rho R_M\\
	&+\frac{1}{2}(R_C\otimes R_C)\mathrm{Alt}\Delta_Cf-(R_C\otimes R_C)\mathrm{Alt}(f\otimes I)\rho\\
	&-\frac{1}{2}(R_C\otimes I+I\otimes R_C+\lambda)\mathrm{Alt}\Delta_CfR_M+(R_C\otimes I+I\otimes R_C+\lambda)\mathrm{Alt}(f\otimes I)\rho R_M\\
	=&-\frac{1}{2}(R_C\otimes R_C)\mathrm{Alt}\Delta_Cf+\frac{1}{2}(I\otimes R_C+R_C\otimes I+\lambda)\mathrm{Alt}\Delta_CfR_M\\
	&+(R_C\otimes R_C)\mathrm{Alt}(f\otimes I)\rho-\mathrm{Alt}(I\otimes R_C+R_C\otimes I+\lambda)(f\otimes I)\rho R_M\\
	&+\frac{1}{2}(R_C\otimes R_C)\mathrm{Alt}\Delta_Cf-(R_C\otimes R_C)\mathrm{Alt}(f\otimes I)\rho\\
	&-\frac{1}{2}(R_C\otimes I+I\otimes R_C+\lambda)\mathrm{Alt}\Delta_CfR_M+(R_C\otimes I+I\otimes R_C+\lambda)\mathrm{Alt}(f\otimes I)\rho R_M\\&=0.
\end{align*}
	For each $h\in C^n(M,C)~(n\geq 2)$, according to Proposition~\ref{alt},
	\begin{align*}
		&\delta^{n+1}\partial^n(h)\\
		=&R_C^{\otimes (n+1)}\partial^nh-\sum^{n}_{i=0}\lambda^{n-i}R_C^{(i)_{n+1}}\partial^nhR_M\\
		=&R_C^{\otimes (n+1)}(\frac{1}{2}\sum_{k=1}^{n}(-1)^k\mathrm{Alt}(\Delta_C\otimes^{(k)}I^{\otimes (n-1)})h+(-1)^{n-1}\mathrm{Alt}(h\otimes I)\rho)\\
		&-\sum^{n}_{i=0}\lambda^{n-i}R_C^{(i)_{n+1}}(\frac{1}{2}\sum_{k=1}^{n}(-1)^k\mathrm{Alt}(\Delta_C\otimes^{(k)}I^{\otimes (n-1)})h+(-1)^{n-1}\mathrm{Alt}(h\otimes I)\rho)R_M\\
		=&\underset{A_1}{\underbrace{\frac{1}{2}\sum_{k=1}^{n}(-1)^k\mathrm{Alt}R_C^{\otimes (n+1)}(\Delta_C\otimes^{(k)}I^{\otimes (n-1)})h}}+\underset{A_2}{\underbrace{(-1)^{n-1}\mathrm{Alt}R_C^{\otimes (n+1)}(h\otimes I)\rho}}\\
		&\underset{A_3}{\underbrace{-\frac{1}{2}\sum_{k=1}^{n}\sum^{n}_{i=0}(-1)^k\lambda^{n-i}\mathrm{Alt}R_C^{(i)_{n+1}}(\Delta_C\otimes^{(k)}I^{\otimes (n-1)})hR_M}}\\&\underset{A_4}{\underbrace{-(-1)^{n-1}\sum^{n}_{i=0}\lambda^{n-i}\mathrm{Alt}R_C^{(i)_{n+1}}(h\otimes I)\rho R_M}},
	\end{align*}
	and
	\begin{align*}
		&\tilde{\partial}^n\delta^n(h)\\
		=&\frac{1}{2}\sum_{k=1}^{n}(-1)^k\mathrm{Alt}((I\otimes R_C+R_C\otimes I+\lambda)\Delta_C\otimes^{(k)}I^{\otimes (n-1)})\delta^nh\\
		&+(-1)^{n-1}\mathrm{Alt}(\delta^nh\otimes R_C)\rho-(-1)^{n-1}\mathrm{Alt}(\delta^nh\otimes I)\rho R_M\\
		=&\frac{1}{2}\sum_{k=1}^{n}(-1)^k\mathrm{Alt}((I\otimes R_C+R_C\otimes I+\lambda)\Delta_C\otimes^{(k)}I^{\otimes (n-1)})(R_C^{\otimes n}h-\sum^{n-1}_{i=0}\lambda^{n-i-1}R_C^{(i)_n}hR_M)\\
		&+(-1)^{n-1}\mathrm{Alt}((R_C^{\otimes n}h-\sum^{n-1}_{i=0}\lambda^{n-i-1}R_C^{(i)_n}hR_M)\otimes R_C)\rho\\&-(-1)^{n-1}\mathrm{Alt}((R_C^{\otimes n}h-\sum^{n-1}_{i=0}\lambda^{n-i-1}R_C^{(i)_n}hR_M)\otimes I)\rho R_M\\
		=&\underset{B_1}{\underbrace{\frac{1}{2}\sum_{k=1}^{n}(-1)^k\mathrm{Alt}R_C^{\otimes (n+1)}(\Delta_C\otimes^{(k)}I^{\otimes (n-1)})h}}\\
		&\underset{B_2}{\underbrace{-\frac{1}{2}\sum_{k=1}^{n}\sum^{n-1}_{i=0}(-1)^k\lambda^{n-i-1}\mathrm{Alt}((I\otimes R_C+R_C\otimes I+\lambda)\Delta_C\otimes^{(k)}I^{\otimes (n-1)})R_C^{(i)_n}hR_M}}\\
		&\underset{B_3}{\underbrace{-(-1)^{n-1}\sum^{n-1}_{i=0}\lambda^{n-i-1}\mathrm{Alt}(R_C^{(i)_n}hR_M\otimes R_C)\rho}}\\&\underset{B_4}{\underbrace{+(-1)^{n-1}\sum^{n-1}_{i=0}\lambda^{n-i-1}\mathrm{Alt}(R_C^{(i)_n}hR_M\otimes I)\rho R_M}}\\
		&+\underset{B_5}{\underbrace{(-1)^{n-1}\mathrm{Alt}(R_C^{\otimes n}h\otimes R_C)\rho}}\underset{B_6}{\underbrace{-(-1)^{n-1}\mathrm{Alt}(R_C^{\otimes n}h\otimes I)\rho R_M}}.
	\end{align*}
	We only need to check that
	\begin{align*}
		A_1+A_2+A_3+A_4=B_1+B_2+B_3+B_4+B_5+B_6.
	\end{align*}
	Since $A_1=B_1$ and $A_2=B_5$, we only need to prove respectively
	\begin{align*}
		A_3&=B_2,\\
		A_4&=B_3+B_4+B_6.
	\end{align*}
Indeed, using (\ref{R}), we can calculate directly
	\begin{align*}
		A_3=&-\frac{1}{2}\sum_{k=1}^{n}\sum^{n-1}_{i=0}(-1)^k\lambda^{n-i}\mathrm{Alt}((I\otimes I)\otimes^{(k)}R_C^{(i)_{n-1}})(\Delta_C\otimes^{(k)}I^{\otimes (n-1)})hR_M\\
		&-\frac{1}{2}\sum_{k=1}^{n}\sum^{n-1}_{i=0}(-1)^k\lambda^{n-i-1}\mathrm{Alt}((R_C\otimes I)\otimes^{(k)}R_C^{(i)_{n-1}})(\Delta_C\otimes^{(k)}I^{\otimes (n-1)})hR_M\\
		&-\frac{1}{2}\sum_{k=1}^{n}\sum^{n-1}_{i=0}(-1)^k\lambda^{n-i-1}\mathrm{Alt}((I\otimes R_C)\otimes^{(k)}R_C^{(i)_{n-1}})(\Delta_C\otimes^{(k)}I^{\otimes (n-1)})hR_M\\
		&-\frac{1}{2}\sum_{k=1}^{n}\sum^{n-2}_{i=0}(-1)^k\lambda^{n-i-2}\mathrm{Alt}((R_C\otimes R_C)\otimes^{(k)}R_C^{(i)_{n-1}})(\Delta_C\otimes^{(k)}I^{\otimes (n-1)})hR_M\\
		=&\underset{A_{31}}{\underbrace{-\frac{1}{2}\sum_{k=1}^{n}\sum^{n-1}_{i=0}(-1)^k\lambda^{n-i}\mathrm{Alt}((I\otimes I)\otimes^{(k)}R_C^{(i)_{n-1}})(\Delta_C\otimes^{(k)}I^{\otimes (n-1)})hR_M}}\\
		&\underset{A_{32}}{\underbrace{-\frac{1}{2}\sum_{k=1}^{n}\sum^{n-1}_{i=0}(-1)^k\lambda^{n-i-1}\mathrm{Alt}((R_C\otimes I)\otimes^{(k)}R_C^{(i)_{n-1}})(\Delta_C\otimes^{(k)}I^{\otimes (n-1)})hR_M}}\\
		&\underset{A_{33}}{\underbrace{-\frac{1}{2}\sum_{k=1}^{n}\sum^{n-1}_{i=0}(-1)^k\lambda^{n-i-1}\mathrm{Alt}((I\otimes R_C)\otimes^{(k)}R_C^{(i)_{n-1}})(\Delta_C\otimes^{(k)}I^{\otimes (n-1)})hR_M}}\\
		&\underset{A_{34}}{\underbrace{-\frac{1}{2}\sum_{k=1}^{n}\sum^{n-2}_{i=0}(-1)^k\lambda^{n-i-2}\mathrm{Alt}((R_C\otimes I+I\otimes R_C+\lambda)\Delta_C\otimes^{(k)}I^{\otimes (n-1)})(R_C\otimes^{(k)} R_C^{(i)_{n-1}})hR_M}},
	\end{align*}
	and
	\begin{align*}
		B_2=&-\frac{1}{2}\sum_{k=1}^{n}\sum^{n-1}_{i=0}(-1)^k\lambda^{n-i-1}\mathrm{Alt}((I\otimes R_C+R_C\otimes I+\lambda)\Delta_C\otimes^{(k)}I^{\otimes (n-1)})R_C^{(i)_n}hR_M\\
		=&\underset{B_{21}}{\underbrace{-\frac{1}{2}\sum_{k=1}^{n}\sum^{n-1}_{i=0}(-1)^k\lambda^{n-i-1}\mathrm{Alt}((I\otimes R_C+R_C\otimes I+\lambda)\Delta_C\otimes^{(k)}I^{\otimes (n-1)})(I\otimes^{(k)} R_C^{(i)_n})hR_M}}\\
		&\underset{B_{22}}{\underbrace{-\frac{1}{2}\sum_{k=1}^{n}\sum^{n-2}_{i=0}(-1)^k\lambda^{n-i-2}\mathrm{Alt}((I\otimes R_C+R_C\otimes I+\lambda)\Delta_C\otimes^{(k)}I^{\otimes (n-1)})(R_C\otimes^{(k)} R_C^{(i)_n})hR_M}}.
	\end{align*}
	Thus, $A_{31}+A_{32}+A_{33}=B_{21}$ and $A_{34}=B_{22}$, that is, $A_3=B_2$.
	\item According to (\ref{rR}),
	\begin{align*}
		A_4=&-(-1)^{n-1}\sum^{n}_{i=0}\lambda^{n-i}\mathrm{Alt}R_C^{(i)_{n+1}}(h\otimes I)\rho R_M\\
		=&-(-1)^{n-1}\sum^{n}_{i=0}\lambda^{n-i}\mathrm{Alt}(R_C^{(i)_n}h\otimes I)\rho R_M
		-(-1)^{n-1}\sum^{n-1}_{i=0}\lambda^{n-i-1}\mathrm{Alt}(R_C^{(i)_n}h\otimes R_C)\rho R_M\\
		=&-(-1)^{n-1}\mathrm{Alt}(R_C^{\otimes n}h\otimes I)\rho R_M-(-1)^{n-1}\sum^{n-1}_{i=0}\lambda^{n-i-1}\lambda \mathrm{Alt}(R_C^{(i)_n}h\otimes I)\rho R_M\\
		&-(-1)^{n-1}\sum^{n-1}_{i=0}\lambda^{n-i-1}\mathrm{Alt}(R_C^{(i)_n}h\otimes R_C)\rho R_M\\
		=&B_6-(-1)^{n-1}\sum^{n-1}_{i=0}\lambda^{n-i-1}\lambda \mathrm{Alt}(R_C^{(i)_n}h\otimes I)\rho R_M
		\\&-(-1)^{n-1}\sum^{n-1}_{i=0}\lambda^{n-i-1}\mathrm{Alt}(R_C^{(i)_n}h\otimes R_C)\rho R_M,
	\end{align*}
	and
	\begin{align*}
		B_3=&-(-1)^{n-1}\sum^{n-1}_{i=0}\lambda^{n-i-1}\mathrm{Alt}(R_C^{(i)_n}h\otimes I)(R_M\otimes R_C)\rho\\
		=&-(-1)^{n-1}\sum^{n-1}_{i=0}\lambda^{n-i-1}\mathrm{Alt}(R_C^{(i)_n}h\otimes I)(R_M\otimes I+I\otimes R_C+\lambda)\rho R_M\\
		=&-B_4+A_4-B_6.
	\end{align*}
	Therefore, $A_4=B_3+B_4+B_6$. The proof is completed.
\end{proof}
Let \begin{equation*}
	C^n_{RB}(M,C)=\left\{
	\begin{aligned}
		&C^0(M,C)=M^*,&n=0,\\
		&C^1(M,C)\oplus\tilde{M}^*,&n=1,\\
		&C^n(M,C)\oplus \tilde{C}^{n-1}(M,C),&n\ge 1.
	\end{aligned}
	\right.
\end{equation*}
Define a linear map $\partial^n_{RB}:C^n_{RB}(M,C)\rightarrow C^{n+1}_{RB}(M,C)$ by
\begin{align*}
	&\partial_{RB}^0(m^*)=(\partial^0(m^*),-\frac{1}{2}\delta^0(m^*)),~~\forall~m^*\in M^*,\\
	&\partial_{RB}^n(f,g)=(\partial^n(f),-\tilde{\partial}^{n-1}(g)-\frac{1}{2}\delta^n(f)),~~\forall~f\in C^n(M,C),~g\in \tilde{C}^{n-1}(M,C).
\end{align*}

\begin{Prop} $\{C^*_{RB}(M,C),\partial^*_{RB}\}$ is a cochain complex , that is,  $$\partial^{n+1}_{RB}\partial^n_{RB}=0~(n\ge 0).$$
\end{Prop}
\begin{proof}
	\begin{enumerate}[label=$(\roman*)$]
		\item When $n=0$, for any $m^*\in M^*$,
		\[\partial^1_{RB}\partial^0_{RB}(m^*)(m)
		=\big(\partial^1(\partial^0(m^*)),-\tilde{\partial}^{0}(-\frac{1}{2}\delta^0(m^*))-\frac{1}{2}\delta^1(\partial^0(m^*))\big)=0.\]
		\item When $n=1$, for any $f\in C^1(M,C)$,
		\[\partial^2_{RB}\partial^1_{RB}(f,m^*)
		=(\partial^2(\partial^1(f)),-\tilde{\partial}^{1}(-\tilde{\partial}^0(m^*)-\frac{1}{2}\delta^1(f))-\frac{1}{2}\delta^2(\partial^1(f)))=0.\]
		\item When $n\ge 2$, for all $f\in C^n(M,C),~g\in C^{n-1}(M,C)$,	\[\partial^{n+1}_{RB}\partial^n_{RB}(f,g)=(\partial^{n+1}\partial^n(f),-\tilde{\partial}^{n}(-\tilde{\partial}^{n-1}(g)-\frac{1}{2}\delta^n(f))-\frac{1}{2}\delta^{n+1}\partial^n(f))=0.\]
	\end{enumerate}
\end{proof}

\begin{Def} The cohomology group of the cochain complex $\{C^*_{RB}(M,C),\partial^*_{RB}\}$ is called the cohomology group of the $\lambda$-weighted Rota-Baxter Lie coalgebra $(C,\Delta_{C},R_{C})$ with coefficients in $(M,\rho,R_M)$. Denote it by \begin{align*}
	\mathrm{H}^n_{RB}(M,C)=\mathrm{Z}^n_{RB}(M,C)/\mathrm{B}^n_{RB}(M,C),
\end{align*}
where
$\mathrm{Z}^n_{RB}(M,C)=\mathrm{ker}(\partial^{n}_{RB}),~
\mathrm{B}^n_{RB}(M,C)=\mathrm{Im}(\partial^{n-1}_{RB}).$
\end{Def}

In the last section, we will need a certain subcomplex of the cochain complex $\{C^*_{RB}(M,C),\partial^*_{RB}\}$ given by
\begin{equation*}
	\bar{C}^n_{RB}(M,C)=\left\{
	\begin{aligned}
		&C^0(M,C)=M^*,&n=0,\\
		&C^1(M,C),&n=1,\\
		&C^n(M,C)\oplus \tilde{C}^{n-1}(M,C),&n\ge 2,
	\end{aligned}
	\right.
\end{equation*}
and $\bar{\partial}^n_{RB}=\partial^n_{RB}|_{\bar{C}^n_{RB}(M,C)}$. The corresponding n-cohomology group is denoted by
$\mathrm{\bar{H}}^n_{RB}(M,C)=\mathrm{\bar{Z}}^n_{RB}(M,C)/\mathrm{\bar{B}}^n_{RB}(M,C)$, which is called the reduced cohomology group
of $(C,\Delta_{C},R_{C})$ with coefficients in $(M,\rho,R_M)$. Obviously, $\mathrm{\bar{H}}^n_{RB}(M,C)=\mathrm{H}^n_{RB}(M,C)$ when $n\ge 3$.

By direct computations,
\begin{enumerate}[label=$(\roman*)$]
	\item $(\mathrm{\bar{H}}^1_{RB})$.
	\begin{align}
		&\mathrm{\bar{B}}^1_{RB}(M,C)=\{\big((m^*\otimes I)\rho,-\frac{1}{2}m^*\big)|m^*\in M^*\},\notag\\&\mathrm{\bar{Z}}^1_{RB}(M,C)=\left\{f\in C^1_{RB}(M,C)\left|\begin{aligned}&(f\otimes I)\rho-(I\otimes f)\tau\rho=\Delta_Cf,\\&fR_M=R_Cf\end{aligned}\right.\right\}.\label{Ccy1}
	\end{align}
	\item $(\mathrm{\bar{H}}^2_{RB})$.
	\begin{align}
		\bar{B}^2_{RB}(M,C)=&\left\{(\mu,\nu)\left|\begin{aligned}&\mu=\frac{1}{2}(f\otimes I)\rho-\frac{1}{2}(I\otimes f)\tau\rho-\frac{1}{2}\Delta_Cf,\\&\nu=-\frac{1}{2}R_Cf+\frac{1}{2}fR_M,f\in \bar{C}^1_{RB}(M,C)\end{aligned}\right.\right\}\label{B2}.
	\end{align}
Since
\begin{align}
&\partial^2(f)\notag\\=&\frac{1}{2}\mathrm{Alt}(I\otimes\Delta_C)h-\frac{1}{2}\mathrm{Alt}(\Delta_C\otimes I)h-\mathrm{Alt}(f\otimes I)\rho,\notag\\
=&(I\otimes\Delta)h+(I\otimes\tau)(\Delta\otimes I)h+(I\otimes\tau)(h\otimes I)\rho-(I\otimes h)\tau\rho \notag
\\&-(\Delta\otimes I)h-(h\otimes I)\rho\notag\\=&0,\label{Z21}\\
&-\tilde{\partial}^1(g)-\frac{1}{2}\delta^2(f)\notag\\
=&\frac{1}{2}\mathrm{Alt}(I\otimes R_C+R_C\otimes I+\lambda)\Delta_C g-\mathrm{Alt}(g\otimes R_C)\rho+\mathrm{Alt}(g\otimes I)\rho R_M\notag\\
&-\frac{1}{2}(R_C\otimes R_C)f+\frac{1}{2}(I\otimes R_C+R_C\otimes I+\lambda)fR_M\notag\\
=&(g\otimes I)\rho R_M-\tau(g\otimes I)\rho R_M+(I\otimes R_C+R_C\otimes I+\lambda)\Delta_C g\notag\\&-(g\otimes R_C)\rho+\tau(g\otimes R_C)\rho-(R_C\otimes R_C)f+(I\otimes R_C+R_C\otimes I+\lambda)fR_M\notag\\=&0.\label{Z22}
\end{align}
Thus,\[\mathrm{\bar{Z}}^2_{RB}(M,C)=\{(f,g)\in \bar{C}^2_{RB}(M,C)|(f,g)~satisfies~Eqs.~(\ref{Z21})-(\ref{Z22})\}.\]
\end{enumerate}

\begin{Theo}
	We have the following long exact sequence of cohomology groups
	\[\xymatrix@1{\cdots\ar[r]&\mathrm{\tilde{H}}^{n-1}(M,C)\ar[r]^-{[i]}&\mathrm{H}^{n}_{RB}(M,C)\ar[r]^-{[p]}&\mathrm{H}^n(M,C)\ar[r]^-{[-\frac{1}{2}\delta^n]}&\mathrm{\tilde{H}}^{n}(M,C)\ar[r]&\cdots}\]
	where \[[g]\in\mathrm{\tilde{H}}^{n-1}(M,C),~~[i][g]=[i(g)]=[(0,g)],\] \[[(f,g)]\in \mathrm{H}^{n}_{RB}(M,C),~~[p][(f,g)]=[p(f,g)]=[f],\]
	\[f\in \mathrm{H}^n(M,C),~~[-\frac{1}{2}\delta^n][f]=[-\frac{1}{2}\delta^n(f)].\]
\end{Theo}
\section{Non-abelian extensions of $\lambda$-weighted Rota-Baxter Lie coalgebras}
In this section, we investigate non-abelian extensions of $\lambda$-weighted Rota-Baxter Lie coalgebras,
define the non-abelian second cohomology groups, and verify that the non-abelian extensions can be classified
by the second non-abelian cohomology groups.
\begin{Def}\label{Def:NAE}
	Let $(C,\Delta_C,R_C)$ and $(M,\Delta_M,R_M)$ be two $\lambda$-weighted Rota-Baxter Lie coalgebras.
 A non-abelian extension of $(C,\Delta_C,R_C)$ by $(M,\Delta_M,R_M)$ is a
 $\lambda$-weighted Rota-Baxter Lie coalgebra $(E,\Delta_E,R_E)$, which fits into a
short exact sequence of $\lambda$-weighted Rota-Baxter Lie coalgebras
	\begin{align}\label{extensible}
		\xymatrix@C=20pt{\mathcal{E}:0\ar[r]&C\ar[r]^-f&E\ar[r]^-g&M\ar[r]&0.}
	\end{align}
When $(M,\Delta_M,R_M)$ is an abelian $\lambda$-weighted Lie coalgebra, the $\mathcal{E}$ is called an abelian extension of $(C,\Delta_C,R_C)$ by $(M,\Delta_M,R_M)$.
Denote an extension as above simply by $(E,\Delta_E,R_E)$ or $\mathcal{E}$.
\end{Def}

A retraction of a non-abelian extension $(E,\Delta_E,R_E)$ of $(C,\Delta_C,R_C)$ by $(M,\Delta_M,R_M)$
 is a linear map $t:E\rightarrow C$ such that $tf = I_{C}$.

\begin{Def}\label{Equ}
	Let $(E_1,\Delta_{E_1},R_{E_1})$ and $(E_2,\Delta_{E_2},R_{E_2})$ be two non-abelian extensions
of $(C,\Delta_C,R_C)$ by $(M,\Delta_M,R_M)$. They are said to be equivalent if there is an isomomorphism $\theta:E_1\rightarrow E_2$
of $\lambda$-weighted Rota-Baxter Lie coalgebras such that the following commutative diagram holds:
	\begin{align}\label{d1} \xymatrix@C=20pt@R=20pt{0\ar[r]&C\ar@{=}[d]\ar[r]^{f_1}&E_1
\ar[d]^\theta\ar[r]^{g_1}&M\ar@{=}[d]\ar[r]&0\\0\ar[r]&C
\ar[r]^{f_2}&E_2\ar[r]^{g_2}&M\ar[r]&0.}
	\end{align}
	We denote the equivalent classes of non-abelian extensions by $\mathrm{Ext}_{nab}(M,C)$.
\end{Def}

\begin{Def} Let $(C,\Delta_C,R_C)$ and $(M,\Delta_M,R_M)$ be two $\lambda$-weighted Rota-Baxter Lie coalgebras.
	A non-abelian 2-cocycle on $(C,\Delta_C,R_C)$ with values in $(M,\Delta_M,R_M)$ is a triple $(h,\rho,\phi)$ of linear maps $h:M\rightarrow C\otimes C,~\rho:M\rightarrow M\otimes C$ and $\phi:M\rightarrow C$, satisfying the following identities:
\begin{equation}
 \tau h+h=0,\label{n0}\end{equation}
\begin{equation}
 (I\otimes\Delta_C)h-(\Delta_C\otimes I)h+(I\otimes\tau)(\Delta_C\otimes I)h
=(h\otimes I)\rho+(I\otimes h)\tau\rho-(I\otimes\tau)(h\otimes I)\rho,\label{n1}\end{equation}
\begin{equation}(I\otimes\Delta_C)\rho+(I\otimes h)\Delta_M=(\rho\otimes I)\rho-(I\otimes\tau)(\rho\otimes I)\rho,\label{n2}\end{equation}
\begin{equation} (\Delta_M\otimes I)\rho=(I\otimes\tau)(\rho\otimes I)\Delta_M+(I\otimes\rho)\Delta_M,\label{n5}\end{equation}
\begin{align}
&(\phi\otimes R_C)\rho-(R_C\otimes\phi)\tau\rho-(\phi\otimes I)\rho R_M+(I\otimes\phi)\tau\rho R_M\notag\\
&+(\phi\otimes\phi)\Delta_M-(I\otimes R_C+R_C\otimes I+\lambda)\Delta_C\phi\notag\\&=(R_C\otimes I+I\otimes R_C+\lambda)hR_M-(R_C\otimes R_C)h,\label{n6}\end{align}
\begin{equation}
 (R_M\otimes\phi)\Delta_M+(R_M\otimes R_C)\rho=(I\otimes\phi)\Delta_MR_M
+(I\otimes R_C+R_M\otimes I+\lambda)\rho R_M.\label{n7}
\end{equation}
\end{Def}

\begin{Def} Two non-abelian 2-cocycles $(h,\rho,\phi)$ and
$(h',\rho',\phi')$ on $(C,\Delta_C,R_C)$ with values in $(M,\Delta_M,R_M)$ are said to be equivalent, if there
exists a linear map $\varphi:M\rightarrow C$
such that
the following equalities hold:
\begin{align}
		h'-h&=(\varphi\otimes I)\rho-\tau(\varphi\otimes I)\rho+(\varphi\otimes\varphi)\Delta_M-\Delta_C\varphi,\label{eqc1}\\
		\rho'-\rho&=(I\otimes \varphi)\Delta_M,\label{eqc2}\\
		\phi'-\phi&=\varphi R_M-R_C\varphi.\label{eqc3}
	\end{align}
\end{Def}	
Denote the set of all non-abelian 2-cocycles on $(C,\Delta_C,R_C)$ with values in $(M,\Delta_M,R_M)$ by $\mathrm{Z}_{nab}^{2}(M,C)$.
The non-abelian second cohomology group $\mathrm{H}^2_{nab}(M,C)$
 is the quotient of $\mathrm{Z}_{nab}^{2}(M,C)$ by this equivalence relation. Denote
the equivalent class of non-abelian 2-cocycle $(h,\rho,\phi)$ by $[(h,\rho,\phi)]$.

Using the above notations, we define a linear map $$\Delta_{(h,\rho,\phi)}:C\oplus M\longrightarrow (C\oplus M)\otimes (C\oplus M)$$ by
\begin{equation}
	\Delta_{(h,\rho)}(c+m)=\Delta_C(c)+h(m)+\Delta_M(m)+\rho(m)-\tau\rho(m),~~\forall~c\in C,m\in M,\label{C1}
\end{equation}
and a linear map $$R_\phi: C\oplus M\longrightarrow C\oplus M$$ by
\begin{equation}
	R_\phi(c+m)=R_C(c)+R_M(m)+\phi(m),~~\forall~c\in C,m\in M.\label{C2}
\end{equation}
\begin{Prop} \label{Pro: LCN}
	 With the above notations,  $(C\oplus M,\Delta_{(h,\rho)},R_\phi)$ is a $\lambda$-weighted
Rota-Baxter Lie coalgebra if and only if $(h,\rho,\phi)$ is a non-abelian 2-cocycles
on $(C,\Delta_C,R_C)$ with values in $(M,\Delta_M,R_M)$. Denote the $\lambda$-weighted
Rota-Baxter Lie coalgebra $(C\oplus M,\Delta_{(h,\rho)},R_\phi)$ simply by $C\oplus_{(h,\rho,\phi)} M$.
\end{Prop}

\begin{proof}
	$(C\oplus M,\Delta_{(h,\rho)},R_\phi)$ is a $\lambda$-weighted
Rota-Baxter Lie coalgebra if and only if Eqs.~(\ref{1.1})-(\ref{R}) hold for $(\Delta_{(h,\rho)},R_\phi)$.
It is easy to be proved that (\ref{1.1}) holds if and only if (\ref{n0}) holds.
 For any $c\in C$, Eqs.~(\ref{LC}), (\ref{R}) hold if and only if $(C,\Delta_C,R_C)$ is a $\lambda$-weighted
Rota-Baxter Lie coalgebra.
According to Eqs.~(\ref{LC}), (\ref{R}) and (\ref{C1}), for all $m\in M$,
\begin{align*}
	&(I\otimes\Delta_{(h,\rho)})\Delta_{(h,\rho)}(m)\\
	=&(I\otimes\Delta_{(h,\rho)})(\Delta_M(m)+\rho(m)-\tau\rho(m)+h(m))\\
	=&(I\otimes(\cancel{\Delta_M(m)}+\rho(m)-\tau\rho(m)+h(m)))\Delta_M(m)+(I\otimes\Delta_C)\rho(m)\\
	&-(I\otimes(\Delta_M(m)+\rho(m)-\tau\rho(m)+h(m)))\tau\rho(m)+(I\otimes\Delta_C)h(m),\\
	&(\Delta_{(h,\rho)}\otimes I)\Delta_{(h,\rho)}(m)\\
	=&(\Delta_{(h,\rho)}\otimes I)(\Delta_M(m)+\rho(m)-\tau\rho(m)+h(m))\\
	=&(\Delta_C\otimes I)h(m)+(\cancel{(\Delta_M(m)}+\rho(m)-\tau\rho(m)+h(m))\otimes I)\Delta_M(m)\\
	&-(\Delta_C \otimes I)\tau\rho(m)+((\Delta_M(m)+\rho(m)-\tau\rho(m)+h(m))\otimes I)\rho(m),\\
	&(I\otimes\tau)(\Delta_{(h,\rho)}\otimes I)\Delta_{(h,\rho)}(m)\\
	=&(I\otimes\tau)(\Delta_C\otimes I)h(m)+(I\otimes\tau)((\cancel{\Delta_M(m)}+\rho(m)-\tau\rho(m)+h(m))\otimes I)\Delta_M(m)\\
	&-(I\otimes\tau)(\Delta_C \otimes I)\tau\rho(m)+(I\otimes\tau)((\Delta_M(m)+\rho(m)-\tau\rho(m)+h(m))\otimes I)\rho(m).
\end{align*}
Thus Eq.~(\ref{LC}) holds if and only if Eqs.~(\ref{n1})-(\ref{n5}) hold. By the same token, Eq.~(\ref{R}) holds for $(\Delta_{(h,\rho)},R_\phi)$
if and only if Eqs.~(\ref{n6}),(\ref{n7}) hold. The proof is completed.
\end{proof}

Let $(C,\Delta_C,R_C)$ and $(M,\Delta_M,R_M)$ be two $\lambda$-weighted Rota-Baxter Lie coalgebras.
Suppose that
\begin{align*}
		\xymatrix@C=20pt{\mathcal{E}:0\ar[r]&C\ar[r]^-f&E\ar[r]^-g&M\ar[r]&0.}
	\end{align*}
is a non-abelian extension of $(C,\Delta_C,R_C)$ by $(M,\Delta_M,R_M)$ with a retraction $t$ of $(E,\Delta_E,R_E)$.
For all $m\in M$, since $g$ is surjective, there exists an element $e\in E$, such that $m=g(e)$.
Define linear maps $h_{t}:M\rightarrow C\otimes C,~\rho_{t}:M\rightarrow M\otimes C$ and $\phi_{t}:M\rightarrow C$ respectively by
\begin{align}
	&h_{t}(m)=h_{t}(g(e))=(t\otimes t)\Delta_E(e)-\Delta_C t(e),\label{2co1}\\
	&\rho_{t}(m)=\rho_{t}(g(e))=(g\otimes t)\Delta_E(e),\label{2co2}\\
	&\phi_{t}(m)=\phi_{t}(g(e))=tR_E(e)-R_Ct(e).\label{2co3}
\end{align}
For all $m\in M$, if $m=g(e_1)=g(e_2),~e_1,e_2\in E$, since $\mathrm{Ker}g= \mathrm{Im}f$, there exists an element
$c\in C$ such that $f(c)=e_1-e_2$. Using $(f\otimes f)\Delta_C=\Delta_C f$ and $tf=I_{C}$, we have
\begin{align*}&((t\otimes t)\Delta_E-\Delta_C t)(e_1-e_2)
\\=&(t\otimes t)\Delta_Ef(c)-\Delta_C tf(c)
\\=&(tf\otimes tf)\Delta_C(c)-\Delta_C (c)
\\=&0,\end{align*}
which implies that $h$ is independent on the choice of $e$. Similarly, we can prove that $\rho,\phi$ are independent on the choice of $e$.

\begin{Prop} \label{CY}
	 With the above notations, $(h_{t},\rho_{t},\phi_{t})$ is a non-abelian 2-cocycles
on $(C,\Delta_C,R_C)$ with values in $(M,\Delta_M,R_M)$. We call it the non-abelian 2-cocycle
corresponding to the extension $\mathcal{E}$ induced by the retraction $t$. Moreover,
 $(C\oplus M,\Delta_{(h_{t},\rho_{t})},R_{\phi_{t}})$ is a $\lambda$-weighted
Rota-Baxter Lie coalgebra. Denote this $\lambda$-weighted
Rota-Baxter Lie coalgebra simply by $C\oplus_{(h_{t},\rho_{t},\phi_{t})} M$.
\end{Prop}
\begin{proof}
It can obtained by direct calculation.
\end{proof}

\begin{lemma}\label{Enc0}
Let $(h_i,\rho_i,\phi_i)$ be the non-abelian 2-cocycle
 corresponding to the extension \begin{align*}
		\xymatrix@C=20pt{\mathcal{E}:0\ar[r]&C\ar[r]^-f&E\ar[r]^-g&M\ar[r]&0.}
	\end{align*}
 induced by retraction $t_i$~(i=1,2).
 Then $(h_1,\rho_1,\phi_1)$ and $(h_2,\rho_2,\phi_2)$
 are equivalent, that is, the equivalent classes of non-abelian 2-cocycles corresponding to
 a non-abelian extension
induced by a retraction are independent on the choice of retractions.
\end{lemma}

\begin{proof}
	Let $(E,\Delta_E,R_E)$ be a non-abelian extension of $(C,\Delta_C,R_C)$ by $(M,\Delta_M,R_M)$.
Suppose that $t_1,t_2$ are different retractions of the extension $\mathcal{E}$,
 $(h_1,\rho_1,\phi_1)$ and $(h_2,\rho_2,\phi_2)$ are the corresponding non-abelian 2-cocycles respectively.
	Since $g$ is surjective, there is an element $e\in E$ for all $m\in M $ such that $m=g(e)$. So we can
define a linear map $\varphi:M\rightarrow C$ by
	\begin{equation}
		\varphi(m)=\varphi(g(e))=t_2(e)-t_1(e),~~\forall~m\in M.\label{C3}
	\end{equation}
In the light of Eqs.~(\ref{2co1})-(\ref{2co3}) and (\ref{C3}), for all $m\in M $, we have
	\begin{align*}
		&h_2(m)-h_1(m)\\=&(t_2\otimes t_2)\Delta_E(e)-\Delta_Ct_2(e)-(t_1\otimes t_1)\Delta_E(e)+\Delta_Ct_1(e)
		\\=&((\varphi g+t_1)\otimes (\varphi g+t_1))\Delta_E(e)-\Delta_C(\varphi g+t_1)(e)-(t_1\otimes t_1)\Delta_E(e)+\Delta_Ct_1(e)
		\\=&(\varphi g\otimes \varphi g)\Delta_E(e)+(t_1\otimes \varphi g)\Delta_E(e)+(\varphi g\otimes t_1)\Delta_E(e)-\Delta_C\varphi(m)
		\\=&(\varphi\otimes \varphi )\Delta_M(m)-\tau(\varphi\otimes I)\rho(m)+(\varphi\otimes I)\rho(m)-\Delta_C\varphi(m).\end{align*}
Analogously, \begin{align*}&\rho_2(m)-\rho_1(m)=(I\otimes\varphi)\Delta_M(m),\\
		&\phi_2(m)-\phi_1(m)=\varphi R_M(m)-R_C\varphi(m).
	\end{align*}
Thus $(h_1,\rho_1,\phi_1)$ and $(h_2,\rho_2,\phi_2)$ are equivalent non-abelian 2-cocycles via a linear map $\varphi$.
\end{proof}

According to Proposition \ref{Pro: LCN} and Proposition \ref{CY}, given a non-abelian extension
 \begin{align*}
		\xymatrix@C=20pt{\mathcal{E}:0\ar[r]&C\ar[r]^-f&E\ar[r]^-g&M\ar[r]&0.}
	\end{align*}
of $(C,\Delta_C,R_C)$ by $(M,\Delta_M,R_M)$ with a retraction $t$, we have a non-abelian 2-cocycle
 $(h_{t},\rho_{t},\phi_{t})$ and a $\lambda$-weighted Rota-Baxter Lie coalgebras $(C\oplus M,\Delta_{(h_{t},\rho_{t})},R_{\phi_{t}})$.
 It follows that
\begin{align*}
\xymatrix@C=20pt{\mathcal{E}_{(h_{t},\rho_{t},\phi_{t})}:0\ar[r]&C\ar[r]^-f&C\oplus_{(h_{t},\rho_{t},\phi_{t})} M\ar[r]^-g&M\ar[r]&0}
	\end{align*}
is a non-abelian extension of $(C,\Delta_C,R_C)$ by $(M,\Delta_M,R_M)$.
Define a linear map
\begin{equation*} \theta:E\longrightarrow C\oplus_{(h_{t},\rho_{t},\phi_{t})}M,~\theta(w)=g(w)+t(w),~~\forall~w\in E.\end{equation*}
We claim that $\theta$ is an isomorphism of $\lambda$-weighted Rota-Baxter Lie coalgebras (the proof can be found in Lemma \ref{E1} ) such that
the following commutative diagram holds:
 \begin{align*} \xymatrix@C=20pt@R=20pt{\mathcal{E}:0\ar[r]&C\ar@{=}[d]\ar[r]^{f}&E
\ar[d]^\theta\ar[r]^{g}&M\ar@{=}[d]\ar[r]&0\\ \mathcal{E}_{(h_{t},\rho_{t},\phi_{t})}:0\ar[r]&C
\ar[r]^-{i}&C\oplus_{(h_{t},\rho_{t},\phi_t)}M\ar[r]^-{\pi}&M\ar[r]&0.}
	\end{align*}
 which indicates that the non-abelian extensions $\mathcal{E}$ and $\mathcal{E}_{(h_{t},\rho_{t},\phi_t)}$
 of $(C,\Delta_C,R_C)$ by $(M,\Delta_M,R_M)$
are equivalent. On the other hand, if $(h,\rho,\phi)$
 is a non-abelian 2-cocycle on $(C,\Delta_C,R_C)$ with values in
$(M,\Delta_M,R_M)$, there is a
$\lambda$-weighted Rota-Baxter Lie coalgebra $(C\oplus M,\Delta_{(h,\rho)},R_{\phi})$, which yields the following
 non-abelian extension of $(C,\Delta_C,R_C)$ by $(M,\Delta_M,R_M)$:
 \begin{align*}
\xymatrix@C=20pt{\mathcal{E}_{(h,\rho,\phi)}:0\ar[r]&C\ar[r]^-i&C\oplus_{(h,\rho,\phi)} M\ar[r]^-\pi&M\ar[r]&0}
	\end{align*}
where $i$ is the inclusion and $\pi$ is the projection.

In the sequel, we characterize the relationship between non-abelian 2-cocycles and non-abelian extensions.

\begin{Theo} \label{Ccy2}
Let $(C,\Delta_C,R_C)$ and $(M,\Delta_M,R_M)$ be two $\lambda$-weighted Rota-Baxter Lie coalgebras.
 Then the equivalent classes of non-abelian extensions of $(C,\Delta_C,R_C)$ by $(M,\Delta_M,R_M)$
are classified by the equivalent classes of non-abelian 2-cocycles. In other words,
\[\mathrm{Ext}_{nab}(M,C)\cong \mathrm{H}^2_{nab}(M,C).\]
\end{Theo}
\begin{proof}
 Define a linear map
 \begin{equation*}\Phi:\mathrm{Ext}_{nab}(M,C)\longrightarrow \mathrm{H}^2_{nab}(M,C),~\end{equation*}
where $\Phi$ assigns an equivalent class of non-abelian extensions to the classes of corresponding non-abelian 2-cocycles.
Firstly, we prove that $\Phi$ is well-defined.
 Assume that two non-abelian extensions $(E_1,\Delta_{E_1},R_{E_1})$ and $(E_2,\Delta_{E_2},R_{E_2})$
of $(C,\Delta_C,R_C)$ by $(M,\Delta_M,R_M)$
 are equivalent via an isomorphism $\theta$, that is, the commutative diagram (\ref{d1}) holds. Let $t_2$ be
a retraction of $(E_2,\Delta_{E_2},R_{E_2})$.
Thanks to $t_2\theta f_1=t_2f_2=I_{C}$, we have $t_1=t_2\theta$ is a retraction of $(E_1,\Delta_{E_1},R_{E_1})$.
Let $(h_1,\rho_1,\phi_1)$ and $(h_2,\rho_2,\phi_2)$ be the corresponding non-abelian 2-cocycles
induced by retractions $t_1,t_2$ respectively. In view of (\ref{2co1}), for any $m\in M$, we get
	\begin{align*}
		h_1(m)&=h_1(g_1(e))=(t_1\otimes t_1)\Delta_{E_1}(e)-\Delta_Ct_1(e)\\&=(t_2\theta\otimes t_2\theta)\Delta_{E_1}(e)-\Delta_Ct_2\theta(e)\\
		&=(t_2\otimes t_2)\Delta_{E_2}\theta(e)-\Delta_Ct_2\theta(e)\\&=h_2(g_2\theta(e))=h_2(g_1(e))=h_2(m).\end{align*}
By the same token,
	    \begin{align*}\rho_1(m)=\rho_2(m),~~
	    \phi_1(m)=\phi_2(m).
\end{align*}
These indicate that $(h_1,\rho_1,\phi_1)=(h_2,\rho_2,\phi_2)$. Thus, $\Phi$ is well-defined.
Secondly, we check that $\Phi$ is bijective. In fact, assume that
$\Phi([\mathcal{E}_1])=[(h_1,\rho_1,\phi_1)]$ and $\Phi([\mathcal{E}_2])=[(h_2,\rho_2,\phi_2)]$.
If $[(h_1,\rho_1,\phi_1)]=[(h_2,\rho_2,\phi_2)]$, we get that the two non-abelian 2-cocycles
$(h_1,\rho_1,\phi_1)$ and $(h_2,\rho_2,\phi_2)$ are equivalent via a linear map $\varphi:M\rightarrow C$, satisfying
 Eqs.~(\ref{eqc1})-(\ref{eqc3}).
 Define a linear map
\begin{equation*}\theta:C\oplus_{(h_1,\rho_1,\phi_1)} M\rightarrow C\oplus_{(h_2,\rho_2,\phi_1)} M \end{equation*}
by
\begin{equation}\label{H1}
	\theta(c+m)=c+\varphi(m)+m,~~\forall~c\in C, m\in M.
\end{equation}
Clearly, $\theta$ is bijective. In the following, we state that $\theta$ is a homomorphism of $\lambda$-weighted Rota-Baxter Lie coalgebras.
In fact, using Eqs.~(\ref{C1}), (\ref{eqc1}), (\ref{eqc2}) and (\ref{H1}), we obtain
 \begin{align*}
	&(\theta\otimes\theta)\Delta_{(h_1,\rho_1)}(c+m)-\Delta_{(h_2,\rho_2)}\theta(c+m)\\
	=&(\theta\otimes\theta)(\Delta_c(c)+h_1(m)+\Delta_M(m)+\rho_1(m)-\tau\rho_1(m))-\Delta_{(h_2,\rho_2)}(c+m+\varphi(m))\\
	=&\Delta_C(c)+h_1(m)+\Delta_M(m)+(\varphi\otimes\varphi)\Delta_M(m)+(I\otimes\varphi)\Delta_M(m)+(\varphi\otimes I)\Delta_M(m)\\
	&+\rho_1(m)+(\varphi\otimes I)\rho_1(m)-\tau\rho_1(m)-\tau(\varphi\otimes I)\rho_1(m)
-\Delta_C(c)-h_2(m)-\Delta_C\varphi(m)\\&-\Delta_M(m)-\rho_2(m)+\tau\rho_2(m)
\\=&h_1(m)-h_2(m)+(\varphi\otimes\varphi)\Delta_M(m)+(\varphi\otimes I)\rho_1(m)-\Delta_C\varphi(m)-\tau(\varphi\otimes I)\rho_1(m)
\\&+\rho_1(m)-\rho_2(m)+(I\otimes\varphi)\Delta_M(m)
-\tau\rho_1(m)+\tau\rho_2(m)+(I\otimes \varphi)\Delta_M(m)
\\=&0,\end{align*}
which yields that
\begin{align*}(\theta\otimes\theta)\Delta_{(h_1,\rho_1)}=\Delta_{(h_2,\rho_2)}\theta.\end{align*}
By Eqs.~ (\ref{eqc3}), (\ref{C2}) and (\ref{H1}), we get
\begin{align*}
	(\theta R_{\phi_1}-R_{\phi_2}\theta)(c+m)&=R_C(c)+R_M(m)+\phi_1(m)+\varphi R_M(m)\\
	&-R_C(c)-R_C\varphi(m)-R_M(m)-\phi_2(m)\\
	&=\phi_1(m)-\phi_2(m)+\varphi R_M(m)-R_C\varphi(m)\\&=0,
\end{align*}
which indicates that $\theta R_{\phi_1}=R_{\phi_2}\theta$.
In all, $\theta$ is an isomorphism of $\lambda$-weighted Rota-Baxter Lie coalgebras. It is easy to check that the following diagram
	\begin{align*} \xymatrix@C=20pt@R=20pt{\mathcal{E}_{(h_1,\rho_1,\phi_1)}:0\ar[r]&C\ar@{=}[d]\ar[r]^-{i_1}&C\oplus_{(h_1,\rho_1,\phi_1)} M
\ar[d]^\theta\ar[r]^-{\pi_1}&M\ar@{=}[d]\ar[r]&0\\ \mathcal{E}_{(h_2,\rho_2,\phi_2)}:0\ar[r]&C
\ar[r]^-{i_2}&C\oplus_{(h_2,\rho_2,\phi_2)} M\ar[r]^-{\pi_2}&M\ar[r]&0}
	\end{align*}
is commutative. Therefore, $[\mathcal{E}_{(h_1,\rho_1,\phi_1)}]=[\mathcal{E}_{(h_2,\rho_2,\phi_2)}]$, which indicates that $\Phi$ is injective.
For any equivalent class of non-abelian 2-cocycles $[(h,\rho)]$, by Proposition \ref{Pro: LCN}, there is
 a non-abelian extension of $(C,\Delta_C,R_C)$ by $(M,\Delta_M,R_M)$:
  \begin{align*}
\xymatrix@C=20pt{\mathcal{E}_{(h,\rho,\phi)}:0\ar[r]&C\ar[r]^-i&C\oplus_{(h,\rho,\phi)} M\ar[r]^-\pi&M\ar[r]&0}
	\end{align*}
   Therefore, $\Phi([\mathcal{E}_{(h,\rho,\phi)}])=[(h,\rho,\phi)]$, which follows that $\Phi$ is surjective.
   In all, $\Phi$ is bijective. This proof is completed.
\end{proof}

\section{Extensibility of a pair of automorphisms}
In this section, we study extensibility of a pair of automorphisms of $\lambda$-weighted Rota-Baxter Lie coalgebras.

In the following, we always suppose that
\[\xymatrix@C=20pt{\mathcal{E}:0\ar[r]&C\ar[r]^-f&E\ar[r]^-g&M\ar[r]&0}\]
is a fixed non-abelian extension of the $\lambda$-weighted Rota-Baxter
Lie coalgebra $(C,\Delta_C,R_C)$ by $(M,\Delta_M,R_M)$, and $t$ is its retraction. Denote $\mathrm{Aut}_{C}(E)=\{\gamma\in \mathrm{Aut}(E)|\gamma (C)=C\}$.
\begin{Def}
	A pair $(\alpha,\beta)\in \mathrm{Aut}(C)\times \mathrm{Aut}(M)$ is said to be extensible with respect to the non-abelian extension $\mathcal{E}$
	if there exists an element $\gamma\in \mathrm{Aut}_{C}(E)$ such that
\begin{equation}\label{AE0}f\alpha=\gamma f,~\beta g=g\gamma,\end{equation}
that is, the following commutative diagram holds:
	\[\xymatrix@C=20pt@R=20pt{0\ar[r]&C\ar[d]^\alpha\ar[r]^f&E\ar[d]^\gamma\ar[r]^g&M
\ar[d]^\beta\ar[r]&0\\0\ar[r]&C\ar[r]^f&E\ar[r]^g&M\ar[r]&0.}\]
\end{Def}

It is natural to ask the following question:

When is a pair of $\lambda$-weighted Rota-Baxter
Lie coalgebra isomorphisms $(\alpha,\beta)\in \mathrm{Aut}(C)\times \mathrm{Aut}(M)$ extensible? We discuss this theme in the sequel.
\begin{lemma} \label{E1}
Let $(h,\rho,\phi)$ be the non-abelian 2-cocycle corresponding to the non-abelian extension $\mathcal{E}$
induced by the retraction $t$. Then the non-abelian extension $\mathcal{E}$ is equivalent to the non-abelian extension
\[\xymatrix@C=20pt{\mathcal{E}_{(h,\rho,\phi)}:0\ar[r]&C\ar[r]^-{i_C}&C\oplus_{(h,\rho,\phi)}\ar[r]^-{\pi_M}&M,\ar[r]&0}\]
where $i_C (resp.~ \pi_{M} )$ is the canonical injection $(resp.~projection)$.
\end{lemma}
\begin{proof}
In the light of $(h,\rho,\phi)$ being a non-abelian 2-cocycle, by Proposition \ref{Pro: LCN},
 $(C\oplus M,\Delta_{(h,\rho)},R_{\phi})$ is a $\lambda$-weighted Rota-Baxter
Lie coalgebra.
Define \begin{equation}\label{Em}\theta:E\longrightarrow C\oplus_{(h,\rho,\phi)} M,~~\hbox{by}~~\theta(e)=t(e)+g(e),\forall~ e\in E.\end{equation}
It is easy to verify that the following diagram is commutative:
\begin{align}
	\xymatrix@C=17pt@R=20pt{\mathcal{E}:0\ar[r]&C\ar@{=}[d]\ar[r]^-f&E
\ar[d]^\theta\ar[r]^-g&M\ar@{=}[d]\ar[r]&0\\ \mathcal{E}_{(h,\rho,\phi)}\ar[r]&C\ar[r]^-{i_C}& C\oplus_{(h,\rho,\phi)}\ar[r]^-{\pi_M}&M\ar[r]&0.}
	\end{align}
By Short Five Lemma, $\theta$ is bijective.
Using Eqs.~(\ref{2co1}), (\ref{2co2}), (\ref{C1}) and (\ref{Em}), we obtain
\begin{eqnarray*}&&(\theta\otimes \theta)\Delta_{E}(e)
\\&=&(t\otimes t)\Delta_{E}(e)+(t\otimes g)\Delta_{E}(e)
+(g\otimes t)\Delta_{E}(e)+(g\otimes g)\Delta_{E}(e)
\\&=&h(g(e))+\Delta_C(t(e))-\tau\rho(g(e))+\rho(g(e))+\Delta_M(g(e))
\\&=&\Delta_{(h,\rho)}(t(e)+g(e))
\\&=&\Delta_{(h,\rho)}\theta(e).
\end{eqnarray*}
By Eqs.~(\ref{2co3}), (\ref{C2}) and (\ref{Em}),
\begin{eqnarray*}R_{\phi}\theta(e)&=&R_{\phi}(t(e)+g(e))
\\&=&R_{C}(t(e))+\phi(g(e))+R_{C}(g(e))
\\&=&tR_{E}(e)+gR_{E}(e)
\\&=&\theta R_{E}(e).
\end{eqnarray*}
Thus, $\theta$ is an isomorphism of $\lambda$-weighted Rota-Baxter Lie coalgebras.
Therefore, the non-abelian extensions $\mathcal{E}_{(h,\rho,\phi)}$ and $\mathcal{E}$ are equivalent via the map $\theta$.
\end{proof}

\begin{lemma}\label{E2} Assume that $\mathcal{E}_1$ and $\mathcal{E}_2$ are two equivalent non-abelian
extensions of $(C,\Delta_C,R_C)$ by $(M,\Delta_M,R_M)$. Let
$(\alpha,\beta)\in \mathrm{Aut} (C)\times\mathrm{Aut}(M)$. Then $(\alpha,\beta)$ is extensible with respect to $\mathcal{E}_1$ if and
only if $(\alpha,\beta)$ is extensible with respect to $\mathcal{E}_2$.
\end{lemma}
\begin{proof} One can take the same procedure of abelian extensions of Lie coalgebras, see
\cite{aut1}.
\end{proof}

\begin{Theo}\label{MT}
	Let $(h,\rho,\phi)$ be a non-abelian 2-cocycle corresponding to the non-abelian extension $\mathcal{E}$
induced by the retraction $t$.
Then $(\alpha,\beta)\in \mathrm{Aut}(C)\times \mathrm{Aut}(M)$ is extensible with respective to $\mathcal{E}$
if and only if there is a linear map $\varphi:M\rightarrow C$ satisfies
\begin{equation}\label{AE1}
		h\beta-(\alpha\otimes\alpha)h=(\varphi\otimes\alpha)\rho-\tau(\varphi\otimes\alpha)\rho
-\Delta_C\varphi+(\varphi\otimes\varphi)\Delta_M,\end{equation}
	\begin{equation}\label{AE2}\rho\beta-(\beta\otimes\alpha)\rho=(\beta\otimes\varphi)\Delta_M,\end{equation}
		\begin{equation}\label{AE3}\phi\beta-\alpha\phi=\varphi R_M-R_C\varphi.
	\end{equation}
\end{Theo}

\begin{proof}
	 Assume that $(\alpha,\beta)$ is extensible with respective to $\mathcal{E}$, then there
 is an automorphism $\gamma\in \mathrm{Aut}_{C}(E)$ such that (\ref{AE0}) holds.
 Since $g$ is surjective, for all $m\in M$, there is an element $e\in E$, such that $m=g(e)$.
 Define a linear map $\varphi:M\rightarrow C$ by
	\begin{equation}\label{AE4}
	\varphi(m)=\varphi g(e)=t\gamma(e)-\alpha t(e),~~\forall~m\in M.
\end{equation}
We should check that $\varphi$ doesn't depend on the choice of $e\in E$.
In fact, for all $m\in M$, if $m=g(e_1)=g(e_2),~e_1,e_2\in E$, due to $\mathrm{Ker}g= \mathrm{Im}f$, there exists an element
$c\in C$ such that $f(c)=e_1-e_2$. Thanks to (\ref{AE0}) and $tf=I_{C}$, we have
$$t\gamma(e_1-e_2)-\alpha t(e_1-e_2)=t\gamma f(c)-\alpha t f(c)=tf\alpha(c)-\alpha (c)=0,$$ which follows that
$\varphi$ doesn't depend on the choice of $e\in E$.
 For all $m\in M$, $\beta(m)=\beta g(e)=g\gamma (e)$.
 Thus, according to Eqs.~(\ref{2co1}), (\ref{2co2}) and (\ref{AE4}),
 we get
\begin{align*}
&h\beta(m)-(\alpha\otimes\alpha)h(m)\\=&(t\otimes t)\Delta_E(\gamma (e))-\Delta_C( t\gamma(e))
-(\alpha t\otimes \alpha t)\Delta_E(e)+(\alpha\otimes\alpha)\Delta_C( t(e))
\\=&(t\gamma\otimes t\gamma)\Delta_E(e)-\Delta_C( t\gamma(e))
-(\alpha t\otimes \alpha t)\Delta_E(e)+\Delta_C (\alpha t(e))
\\=&\Big((\varphi g+\alpha t)\otimes (\varphi g+\alpha t)\Big)\Delta_E(e)-\Delta_C( (\varphi g+\alpha t)(e))
-(\alpha t\otimes \alpha t)\Delta_E(e)+\Delta_C (\alpha t(e))
\\=&(\varphi g\otimes \varphi g)\Delta_E(e)+(\alpha t\otimes \alpha t)\Delta_E(e)
+(\varphi g\otimes \alpha t)\Delta_E(e)+(\alpha t\otimes \varphi g)\Delta_E(e)\\&-\Delta_C( \varphi g(e))-\Delta_C(\alpha t(e))-(\alpha t\otimes \alpha t)\Delta_E(e)
+\Delta_C (\alpha t(e))
\\=&(\varphi \otimes \varphi )\Delta_M(m)
+(\varphi \otimes \alpha )\rho(m)-\tau(\varphi \otimes \alpha )\rho(m)-\Delta_C( \varphi (m))
,\end{align*}
 which yields that (\ref{AE1}) holds. Analogously, we can check that Eqs.~(\ref{AE2})- (\ref{AE3}) hold.
 Conversely, in the light of Lemma \ref{E1} and Lemma \ref{E2},
  we only need to prove that $(\alpha,\beta)\in \mathrm{Aut}(C
)\times \mathrm{Aut}(M)$ is extensible with respect to the non-abelian extension
\[\xymatrix@C=20pt{\mathcal{E}_{(h,\rho,\phi)}:0\ar[r]&C\ar[r]^-{i_C}&C\oplus_{(h,\rho,\phi)} \ar[r]^-{\pi_M}&M\ar[r]&0.}\]
 In fact, take
 $\gamma=\begin {bmatrix}
 \alpha&\varphi \\
0&\beta
\end {bmatrix}$, that is,
 \begin{equation}\label{AE5}\gamma(c+m)=\alpha(c)+\varphi(m)+\beta(m),~~\forall~c\in C, m\in M.\end{equation}
 It is clear that
$\gamma$ is a bijection, $\gamma|_{C}=\alpha,~i_{C}\alpha=\gamma i_{C}$ and $\pi_{M} \gamma=\beta\pi_{M}$.
In the remaining part, we only need to verify that $\gamma$ is a homomorphism of $\lambda$-weighted Rota-Baxter
Lie coalgebras.

For all $c\in C, m\in M$, in view of Eqs.~(\ref{C1}), (\ref{AE1}), (\ref{AE2}) and (\ref{AE5}), by direct computations, we have
		\begin{align*}(\gamma\otimes\gamma)\Delta_{(h,\rho)}(c+m)&=(\gamma\otimes\gamma)(\Delta_C(c)+h(m)+\rho(m)-\tau\rho(m)+\Delta_M(m))
		\\=&(\alpha\otimes\alpha)\Delta_C(c)+(\alpha\otimes\alpha)h(m)+(\beta\otimes\alpha)\rho(m)+(\varphi\otimes\alpha)\rho(m)\\
		&-\tau(\beta\otimes\alpha)\rho(m)-\tau(\varphi\otimes\alpha)\rho(m)+(\beta\otimes\beta)\Delta_M(m)\\
		&+(\beta\otimes\varphi)\Delta_M(m)+(\varphi\otimes\beta)\Delta_M(m)+(\varphi\otimes\varphi)\Delta_M(m)
\\=&\Delta_C(\alpha(c))+\Delta_C(\varphi(m))+(h+\rho-\tau\rho)(\beta(m))+
		\Delta_M(\beta(m))
\\=&\Delta_{(h,\rho)}(\alpha(c)+\varphi(m)+\beta(m))
\\=&\Delta_{(h,\rho)}\gamma(c+m),
\end{align*}
which implies that $\Delta_{(h,\rho)}\gamma=(\gamma\otimes\gamma)\Delta_{(h,\rho)}$.
By the same token, $R_\phi\gamma=\gamma R_\phi$.
Hence $\gamma$ is a homomorphism of $\lambda$-weighted Rota-Baxter
Lie coalgebras. In all, $(\alpha,\beta)\in \mathrm{Aut}(C
)\times \mathrm{Aut}(M)$ is extensible with respect to the non-abelian extension
$\mathcal{E}_{(h,\rho,\phi)}$. We complete the proof.
\end{proof}

Let $(h,\rho,\phi)$ be a non-abelian 2-cocycle corresponding to the non-abelian extension $\mathcal{E}$
induced by the retraction $t$.
For any $(\alpha,\beta)\in \mathrm{Aut}(C)\times \mathrm{Aut}(M)$, we define a
triple $(h_{(\alpha,\beta)},\rho_{(\alpha,\beta)},\phi_{(\alpha,\beta)})$ of linear maps
$h_{(\alpha,\beta)}:M\rightarrow C\otimes C,~\rho_{(\alpha,\beta)}:M\rightarrow M\otimes C$ and $\phi_{(\alpha,\beta)}:M\rightarrow C$
respectively by
\begin{equation}\label{NCA1} h_{(\alpha,\beta)}=(\alpha\otimes\alpha)h\beta^{-1},~\rho_{(\alpha,\beta)}
=(\beta\otimes\alpha)\rho\beta^{-1},~\phi_{(\alpha,\beta)}=\alpha\phi\beta^{-1}.\end{equation}
	
\begin{Prop} \label{EEnc0}
	With the above notations, the triple $(h_{(\alpha,\beta)},\rho_{(\alpha,\beta)},\phi_{(\alpha,\beta)})$ is a non-abelian 2-cocycle.
\end{Prop}
\begin{proof}
	According to Eqs.~(\ref{n1}) and (\ref{NCA1}),
	\begin{align*}
		&(I\otimes\Delta)h_{(\alpha,\beta)}-(I\otimes h_{(\alpha,\beta)})\tau\rho_{(\alpha,\beta)}
-(\Delta\otimes I)h_{(\alpha,\beta)}-(h_{(\alpha,\beta)}\otimes I)\rho_{(\alpha,\beta)}\\
		&+(I\otimes\tau)(\Delta\otimes I)h_{(\alpha,\beta)}+(I\otimes\tau)(h_{(\alpha,\beta)}\otimes I)\rho_{(\alpha,\beta)}\\
		=&(I\otimes\Delta)(\alpha\otimes\alpha)h\beta^{-1}-(I\otimes (\alpha\otimes\alpha)h\beta^{-1})\tau(\beta\otimes\alpha)\rho\beta^{-1}-(\Delta\otimes I)(\alpha\otimes\alpha)h\beta^{-1}\\&-((\alpha\otimes\alpha)h\beta^{-1}\otimes I)(\beta\otimes\alpha)\rho\beta^{-1}+(I\otimes\tau)(\Delta\otimes I)(\alpha\otimes\alpha)h\beta^{-1}\\&+(I\otimes\tau)((\alpha\otimes\alpha)h\beta^{-1}\otimes I)(\beta\otimes\alpha)\rho\beta^{-1}\\
		=&(\alpha\otimes\Delta\alpha)h\beta^{-1}-(\alpha\otimes (\alpha\otimes\alpha)h)\tau\rho\beta^{-1}-(\Delta\alpha\otimes \alpha)h\beta^{-1}\\&-((\alpha\otimes\alpha)h\otimes \alpha)\rho\beta^{-1}+(I\otimes\tau)(\Delta\alpha\otimes \alpha)h\beta^{-1}+(I\otimes\tau)((\alpha\otimes\alpha)h\otimes \alpha)\rho\beta^{-1}\\
		=&(\alpha\otimes\alpha\otimes\alpha)(I\otimes\Delta)h\beta^{-1}-(\alpha\otimes\alpha\otimes\alpha)(I\otimes h)\tau\rho\beta^{-1}-(\alpha\otimes\alpha\otimes\alpha)(\Delta\otimes I)h\beta^{-1}\\
		&-(\alpha\otimes\alpha\otimes\alpha)(h\otimes I)\rho\beta^{-1}+(\alpha\otimes\alpha\otimes\alpha)(I\otimes \tau)(\Delta\otimes I)h\beta^{-1}\\&+(\alpha\otimes\alpha\otimes\alpha)(I\otimes\tau)(h\otimes I)\rho\beta^{-1}\\
		=&(\alpha\otimes\alpha\otimes\alpha)((I\otimes\Delta)h-(I\otimes h)\tau\rho-(\Delta\otimes I)h-(h\otimes I)\rho+(I\otimes\tau)(\Delta\otimes I)h\\
		&+(I\otimes\tau)(h\otimes I)\rho)\beta^{-1}\\=&0,
	\end{align*}
which yields that Eq.~(\ref{n1}) holds for $(h_{(\alpha,\beta)},\rho_{(\alpha,\beta)})$.
To check that Eqs.~(\ref{n2})-(\ref{n7}) holds
for $(h_{(\alpha,\beta)},\rho_{(\alpha,\beta)},\phi_{(\alpha,\beta)})$, one can take the same procedure.
Therefore, $(h_{(\alpha,\beta)},\rho_{(\alpha,\beta)},\phi_{(\alpha,\beta)})$ is a non-abelian 2-cocycle.
\end{proof}

\begin{Theo}\label{EEnc} Let $(h,\rho,\phi)$ be a non-abelian 2-cocycle corresponding to the non-abelian
extension $\mathcal{E}$
induced by the retraction $t$.
Then $(\alpha,\beta)\in \mathrm{Aut}(C)\times\mathrm{Aut}(M)$ is extensible with respect to $\mathcal{E}$
if and only if
$(h_{(\alpha,\beta)},\rho_{(\alpha,\beta)},\phi_{(\alpha,\beta)})$ and $(h,\rho,\phi)$ are equivalent non-abelian 2-cocycles.
\end{Theo}
\begin{proof}
Suppose that $(\alpha,\beta)\in \mathrm{Aut}(C)\times \mathrm{Aut}(M)$
is extensible with respect to $\mathcal{E}$, by Theorem \ref{MT}, there is a linear map
$\varphi:M\rightarrow C$ satisfying
Eqs.~(\ref{AE1})-(\ref{AE3}). For all $m\in M$, there is $m_0\in M$ such that $m=\beta(m_0)$. Since $g$ is surjective,
there exist elements $e,e_0\in E$, such that $m=g(e),m_0=g(e_0)$.
Combining Eqs.~(\ref{AE1})-(\ref{AE2}) and (\ref{NCA1}),
	\begin{align*}
		&h(m)-h_{(\alpha,\beta)}(m)\\=&h(m)-(\alpha\otimes\alpha)h(\beta^{-1}(m))\\
		=&h(\beta(m_0))-(\alpha\otimes\alpha)h(m_0)\\
		=&(\varphi\otimes\alpha)\rho(m_0)-\tau(\varphi\otimes\alpha)\rho(m_0)
-\Delta_C(\varphi(m_0))+(\varphi\otimes\varphi)\Delta_M(m_0)
\\=&\Big(\varphi\otimes\alpha-\tau(\varphi\otimes\alpha)\Big)\rho(\beta^{-1}(m_0))
-\Delta_C(\varphi\beta^{-1}(m_0))+(\varphi\otimes\varphi)\Delta_M(\beta^{-1}(m_0))
\\=&(\varphi\beta^{-1}\otimes I)(\beta\otimes\alpha)\rho(\beta^{-1}(m))-\tau(\varphi\beta^{-1}\otimes I)(\beta\otimes\alpha)\rho(\beta^{-1}(m))
\\&-\Delta_C(\varphi\beta^{-1}(m))+(\varphi \beta^{-1}\otimes\varphi\beta^{-1})\Delta_M(m)
\\=&\Big(\varphi\beta^{-1}\otimes I-\tau(\varphi\beta^{-1}\otimes I)\Big)\rho_{(\alpha,\beta)}(m)
-\Delta_C(\varphi\beta^{-1}(m))+(\varphi \beta^{-1}\otimes\varphi\beta^{-1})\Delta_M(m).
	\end{align*}
By the same token,
\begin{align*}&\rho(m)-\rho_{(\alpha,\beta)}(m)=(I\otimes \varphi\beta^{-1})\Delta_M(m),\\&\phi(m)-\phi_{(\alpha,\beta)}(m)=(\varphi \beta^{-1})R_M-R_{C}(\varphi \beta^{-1}).\end{align*}
Thus, $(h,\rho,\phi)$ and $(h_{(\alpha,\beta)},\rho_{(\alpha,\beta)},\phi_{(\alpha,\beta)})$ are equivalent non-abelian 2-cocycles via a linear map $\varphi\beta^{-1}$.
The converse part can be checked similarly.
\end{proof}

\section{Wells exact sequences for $\lambda$-weighted Rota-Baxter
Lie coalgebras}

In this section, we always suppose that
\[\xymatrix@C=20pt{\mathcal{E}:0\ar[r]&C\ar[r]^f&E\ar[r]^g&M\ar[r]&0}\]
is a fixed non-abelian extension of the $\lambda$-weighted Rota-Baxter
Lie coalgebra $(C,\Delta_C,R_C)$ by $(M,\Delta_M,R_M)$, and $t$ is its retraction.
Then there is a linear map $s:M\rightarrow E$ such that
\begin{equation}\label{W0}ft+sg=I_{E}.\end{equation}

Let $(h,\rho,\phi)$ be a non-abelian 2-cocycle corresponding to the non-abelian extension $\mathcal{E}$
induced by the retraction $t$.
Define a map $W:\mathrm{Aut}(C)\times \mathrm{Aut}(M)\rightarrow \mathrm{H}^2_{nab}(M,C)$ by
\begin{equation}\label{W1}
	W(\alpha,\beta)=[(h_{(\alpha,\beta)},\rho_{(\alpha,\beta)},\phi_{(\alpha,\beta)})-(h,\rho,\phi)].
\end{equation}
The map $W$ is called the Wells map.

\begin{Prop}\label{Wm1}
	The Well maps $W$ does not depend on the choice of retractions.
\end{Prop}
\begin{proof} For all $m\in M$, there is a $m_0\in M$ such that $m=\beta(m_0)$. Since $g$ is surjective,
there exist elements $e,e_0\in E$, such that $m=g(e),m_0=g(e_0)$.
Assume that $t'$ is another retraction and $(h',\rho',\phi')$ is the corresponding induced non-abelian 2-cocycle.
Then $(h',\rho',\phi')$ and $(h,\rho,\phi)$ are equivalent non-abelian 2-cocycles via a linear map
 \begin{equation}\label{W2}\varphi(m)=\varphi (g(e))=t'(e)-t(e),~~\forall~m\in M.\end{equation}
 Using Eqs.~(\ref{2co1})-(\ref{2co2}), (\ref{NCA1}) and (\ref{W2}), we have
	\begin{align*}
		&h'_{\alpha,\beta}(m)-h_{\alpha,\beta}(m)\\=&(\alpha\otimes\alpha)h'(\beta^{-1}(m))-(\alpha\otimes\alpha)h(\beta^{-1}(m))\\
		=&(\alpha\otimes\alpha)h'(m_0)-(\alpha\otimes\alpha)h(m_0)\\
		=&(\alpha\otimes\alpha)h'(g(e_0))-(\alpha\otimes\alpha)h(g(e_0))\\
		=&(\alpha\otimes\alpha)((t'\otimes t')\Delta_E(e_0)-\Delta_C t'(e_0)-(t\otimes t)\Delta_E(e_0)+\Delta_C t(e_0))\\
		=&(\alpha\otimes\alpha)\Big(((\varphi g+t)\otimes (\varphi g+t))\Delta_E(e_0)-(t\otimes t)\Delta_E(e_0)-\Delta_C \varphi g(e_0)\Big)\\
		=&(\alpha\otimes\alpha)((\varphi g\otimes \varphi g)\Delta_E(e_0)+(\varphi g\otimes t)\Delta_E(e_0)+(t\otimes \varphi g)\Delta_E(e_0)-\Delta_C \varphi g(e_0))\\=&(\alpha\varphi\otimes \alpha\varphi)\Delta_M(m_0)+(\alpha\varphi \otimes \alpha )\rho(m_0)
-(\alpha \otimes \alpha\varphi )\tau\rho(m_0)-\Delta_C(\alpha \varphi (m_0))\\
		=&(\alpha\varphi\beta^{-1}\otimes \alpha\varphi\beta^{-1})\Delta_M(m)+\Big((\alpha\varphi \otimes \alpha )-(\alpha\varphi \otimes \alpha )\tau \Big)\rho\beta^{-1}(m)-\Delta_C(\alpha \varphi \beta^{-1}(m))\\
		=&(\alpha\varphi\beta^{-1}\otimes \alpha\varphi\beta^{-1})\Delta_M(m)+\Big((\alpha\varphi\beta^{-1} \otimes I )-\tau(\alpha\varphi\beta^{-1} \otimes I ) \Big)\rho_{(\alpha,\beta)}(m)\\&-\Delta_C(\alpha \varphi \beta^{-1}(m)).
	\end{align*}
Analogously,
\begin{align*}
		&\rho'_{(\alpha,\beta)}(m)-\rho_{(\alpha,\beta)}(m)=(I\otimes \alpha\varphi\beta^{-1})\Delta_M(m),
		\\&\phi'_{(\alpha,\beta)}(m)-\phi_{(\alpha,\beta)}(m)=(\alpha\varphi\beta^{-1}) R_M(m)-R_{C}(\alpha\varphi\beta^{-1})(m).\end{align*}
Thus, $(h'_{(\alpha,\beta)},\rho'_{(\alpha,\beta)},\phi'_{(\alpha,\beta)})$ and $(h_{(\alpha,\beta)},\rho_{(\alpha,\beta)},\phi_{(\alpha,\beta)})$
 are equivalent non-abelian 2-cocycles via the linear map $\alpha\varphi\beta^{-1}$.
Combining Lemma \ref{Enc0}, we get that
 $(h'_{(\alpha,\beta)},\rho'_{(\alpha,\beta)},\phi'_{(\alpha,\beta)})-(h',\rho',\phi')$ and $(h_{(\alpha,\beta)},\rho_{(\alpha,\beta)},\phi_{(\alpha,\beta)})-(h,\rho,\phi)$ are equivalent via the linear map $\alpha\varphi\beta^{-1}-\varphi$.
\end{proof}

\begin{lemma}\label{Wm2} The following linear map $K$ is well defined:
\begin{equation}\label{W3}K:\mathrm{Aut}_{C}(E)\longrightarrow \mathrm{Aut}(C)\times \mathrm{Aut}(M),~K(\gamma)=(\alpha,\beta),~~\forall~\gamma\in \mathrm{Aut}_{C}(E),\end{equation}
where
\begin{equation}\label{W4}\alpha(c)=t\gamma f(c),~\beta(m)=g\gamma(e),~~\forall~c\in C,m\in M \hbox{ and } m=g(e),e\in E.\end{equation}
\end{lemma}

\begin{proof}
It is similar to Lemma 6.1 \cite{aut1}.
\end{proof}

\begin{Theo}\label{Wm3} There is an exact sequence:
	\[\xymatrix@C=20pt{1\ar[r]&\mathrm{Aut}_{C}^{M}(E)\ar[r]^-T&\mathrm{Aut}_{C}(E)\ar[r]^-K&\mathrm{Aut}(C)\times \mathrm{Aut}(M)\ar[r]^-W&\mathrm{H}^2_{nab}(M,C)}\]
where $\mathrm{Aut}_{C}^{M}(E)=\{\gamma \in \mathrm{Aut}_C(E)| K(\gamma)=(I_{C},I_{M}) \}$.
\end{Theo}
\begin{proof} It is obviously, $\mathrm{Ker} K=\mathrm{Im}T$ and $T$ is injective. We only need to prove that $\mathrm{Ker} W=\mathrm{Im}K$.
 In view of Theorem \ref{EEnc}, for all $(\alpha,\beta)\in \mathrm{Ker} W$,
 we know that $(\alpha,\beta)$ is extensible with respect to the non-abelian extension $\mathcal{E}$, that is,
  there is a $\gamma\in \mathrm{Aut}_{C}^{M}(E)$, such that
$f\alpha=\gamma f,~\beta g=g\gamma$, which follows that $\alpha=tf\alpha=t\gamma f,~\beta (m)=\beta g(e)=g\gamma(e).$
Thus,  $(\alpha,\beta)\in \mathrm{Im}K$. On the other hand, for any $(\alpha,\beta)\in \mathrm{Im}K$, there is an isomorphism
$\gamma\in \mathrm{Aut}_{C}(E)$, such that
 (\ref{W4}) holds. Combining (\ref{W0}) and $\mathrm{Im}f=\mathrm{Ker }g$, we have
 $f\alpha=ft\gamma f=(I_{E}-sg)\gamma f=\gamma f$ and $\beta g=g\gamma$.
 Thus, $(\alpha,\beta)$ is extensible with respect to the non-abelian extension $\mathcal{E}$. According to Theorem \ref{EEnc},
  $(\alpha,\beta)\in \mathrm{Ker} W$. So $\mathrm{Ker} W=\mathrm{Im}K$.
\end{proof}

Suppose that
\begin{align}\mathrm{Z}_{nab}^{1}(M,C)&=\left\{\varphi:M\rightarrow C\left|\begin{aligned}
		&(\varphi\otimes I)\rho-\tau(\varphi\otimes I)\rho
		=\Delta_C\varphi-(\varphi\otimes\varphi)\Delta_M,\\&(I\otimes\varphi)\Delta_M=0,\varphi R_M=R_C\varphi
	\end{aligned}\right.\right\}\label{W5}.\end{align}
It is easy to check that $\mathrm{Z}_{nab}^{1}(M,C)$ is an abelian group, which is called a non-abelian 1-cocycle.

\begin{Prop}\label{Wm4} (i) The linear map $\chi:\mathrm{Ker K}\rightarrow \mathrm{Z}_{nab}^{1}(M,C)$ given by
$$\chi(\gamma)=\varphi_{\gamma},~~\forall~\gamma\in \mathrm{Ker} K,$$ is
a homomorphism of groups, where
\begin{equation}\label{W6}\varphi_{\gamma}(m)=t\gamma(e)-t(e),~~\forall~m\in M, g(e)=m ~\hbox{for some}~ e\in E.\end{equation}

(ii) $\chi$ is an isomorphism, that is, $\mathrm{Ker K}\cong \mathrm{Z}_{nab}^{1}(M,C)$.
\end{Prop}

\begin{proof} (i) For all $m\in M$, if there are $e_1,e_2\in E$ such that $m=g(e_1)=g(e_2)$.
 Due to $\mathrm{Ker} g=\mathrm{Im}f$, there is a $c\in C$
satisfying $f(c)=e_1-e_2$, then
\begin{align*}t\gamma(e_1-e_2)-t(e_1-e_2)=t\gamma f(c)-tf(c)=0.\end{align*}
Thus, $\varphi_{\gamma}$ is independent on the choice of $e$.
 Using Eqs.~(\ref{2co1})-(\ref{2co2}), (\ref{W4}) and (\ref{W6}), for all $m\in M$, we have
\begin{align*}&(\varphi_{\gamma}\otimes I)\rho(m)-\tau(\varphi_{\gamma}\otimes I)\rho(m)
-\Delta_C(\varphi_{\gamma}(m))+(\varphi_{\gamma}\otimes\varphi_{\gamma})\Delta_M(m)\\=&
(\varphi_{\gamma} g\otimes t)\Delta_E(e)-\tau(\varphi_{\gamma} g\otimes t)\Delta_E(e)-\Delta_C(\varphi_{\gamma} g(e))
+(\varphi_{\gamma}\otimes\varphi_{\gamma})\Delta_M(g(e))
\\=&
(t\gamma\otimes t)\Delta_E(e)-(t\otimes t)\Delta_E(e)
+(t\otimes t\gamma)\Delta_E(e)-(t\otimes t)\Delta_E(e)\\&-\Delta_C(t\gamma(e))+\Delta_C(t(e))
+( t\gamma\otimes t\gamma)\Delta_E(e)-( t\gamma\otimes t)\Delta_E(e)\\&-( t\otimes t\gamma)\Delta_E(e)
+( t\otimes t)\Delta_E(e)
\\=&
\Delta_C(t(e))-(t\otimes t)\Delta_E(e)+( t\gamma\otimes t\gamma)\Delta_E(e)
-\Delta_C(t\gamma(e))
\\=&h(g\gamma(e))
-h(g(e))\\=&0.
\end{align*}
Analogously, $(I\otimes\varphi_{\gamma})\Delta_M=0,~~\varphi_{\gamma} R_M-R_C\varphi_{\gamma}=0$. Therefore, $\varphi_{\gamma}\in \mathrm{Z}_{nab}^{1}(M,C)$
and $\chi$ is well-defined.
For any $\gamma_1,\gamma_2\in \mathrm{Ker} K$ and $m\in M$, we obtain
\begin{align*}\chi(\gamma_1 \gamma_2)(m)&=t\gamma_1 \gamma_2(e)-t(e)
\\&=(\varphi_{\gamma_1 }g+t)\gamma_2(e)-t(e)
\\&=\varphi_{\gamma_1 }g(e)+\varphi_{\gamma_2 }(m)
\\&=\varphi_{\gamma_1 }(m)+\varphi_{\gamma_2 }(m)
\\&=\chi(\gamma_1 )(m)+\chi(\gamma_2 )(m),\end{align*}
which means that $\chi$ is a homomorphism of groups.

(ii) For all $\gamma\in \mathrm{Ker}K$, then we have $K (\gamma)=0$, that is, $t\gamma f=I_{C}, g=g\gamma.$
 It follows that $\gamma(e)-e\in \mathrm{Ker}g$.
Combining $\mathrm{Imf}= \mathrm{Ker}g$, there is an element $c\in C$ such that $\gamma(e)-e=f(c)$.
If $\chi(\gamma)=0$, then $\chi(\gamma)(m)=t\gamma(e)-t(e)=0.$
Thus, $c=tf(c)=t\gamma(e)-t(e)=0$. Then
 $\gamma(e)-e=0$ and thus $\gamma=I_E$, which indicates that $\chi$ is injective.
 Secondly, we check that $\chi$ is surjective.
 For any $\varphi\in \mathrm{Z}_{nab}^{1}(M,C)$, give a linear map $\gamma:E\rightarrow E$ by
  \begin{equation}\label{W7}\gamma(e)=f\varphi g(e)+e,~~\forall~e\in E.\end{equation}
Then $\gamma$ is a homomorphism of $\lambda$-weighted Rota-Baxter Lie coalgebras. Indeed,
by Eqs.~(\ref{2co2}),(\ref{W0}) and (\ref{W5}), we have for all $e\in E$,
  \begin{align}&(f\varphi g\otimes I )\Delta_{E}(e)\notag\\
  =&(f\varphi g\otimes ft)\Delta_{E}(e)+(f\varphi g\otimes sg)\Delta_{E}(e)\notag\\
  =&(f\otimes f)(\varphi \otimes I)( g\otimes t)\Delta_{E}(e)+(f\otimes s)(\varphi\otimes I)\Delta_{M}(g(e))
  \notag\\
  =&(f\otimes f)(\varphi \otimes I)\rho(g(e))+(f\otimes s)(\varphi\otimes I)\Delta_{M}(g(e))
  \notag\\
  =&(f\otimes f)(\varphi \otimes I)\rho(g(e)),\label{W8}
 \end{align}
 and by the same token,
 \begin{equation}\label{W9}(I\otimes f\varphi g )\Delta_{E}(e)=-\tau(f\otimes f)(\varphi \otimes I)\rho(g(e)).\end{equation}
  Using Eqs.~(\ref{W7})-(\ref{W9}), we have for all $e\in E$,
  \begin{align*}
		& (\gamma \otimes \gamma)\Delta_{E}(e)
\\=&(f\varphi g\otimes I)\Delta_{E}(e)
 +(I\otimes f\varphi g)\Delta_{E}(e)+(I\otimes I)\Delta_{E}(e) +( f\varphi g\otimes f\varphi g)\Delta_{E}(e)
  \\=&(f\otimes f)(\varphi \otimes I)\rho(g(e))-\tau(f\otimes f)(\varphi \otimes I)\rho(m)
   +( f\varphi \otimes f\varphi )\Delta_{M}(g(e))+\Delta_{E}(e)
 \\=&(f\otimes f)\Delta_{C}\varphi(g(e))+\Delta_{E}(e)
 \\=&\Delta_{C}(f\varphi g(e)+e)
\\=& \Delta_{E}(\gamma(e)).
\end{align*}
 Analogously, $\gamma R_E=R_E\gamma$. Therefore, $\gamma$ is a homomorphism of $\lambda$-weighted Rota-Baxter Lie coalgebras.
 In the sequel, we state that $\gamma $ is bijective. If $\gamma(e)=f\varphi g(e)+e=0$, then
$0=ftf\varphi g(e)+e=-ft(e)+e$. Combining $gf=0$, we get $f\varphi g(e)=f\varphi gft(e)=0$, which follows that $e=-f\varphi g(e)=0$.
So $\gamma$ is injective. For all $e\in E$, due to $gf=0$, we have
$\gamma(e-f\varphi g(e))=e-f\varphi g(e)+f\varphi g(e-f\varphi g(e))=e$, which yields that $\gamma$ is bijective.
 Since $gf=0$, $\gamma f(c)=f\varphi gf(c)+f(c)=f(c),~\forall~c\in C$. In all, $\gamma\in \mathrm{Aut}_{C}(E)$.
 Combining $gf=0,tf=I_C$ and (\ref{W7}), for all $c\in C,m\in M$, we have
 \begin{align*}
		& \alpha(c)=t\gamma f(c)=t(f\varphi g f(c)+f(c))=tf(c)=c,
\\&\beta(m)=g\gamma(e)=g(f\varphi g(e)+e)=g(e)=m,
\end{align*}
 which imply that $\alpha=I_C,\beta=I_M$, thus $\gamma\in \mathrm{Ker}K$.
 Therefore, $\chi$ is bijective.
 So $\mathrm{Ker} K\simeq \mathrm{Z}_{nab}^{1}(M,C)$.
\end{proof}
Combining Theorem \ref{Wm3} and Proposition \ref{Wm4}, we have
\begin{Theo}\label{Wm5} There is an exact sequence:
\[\xymatrix@C=20pt{0\ar[r]&\mathrm{Z}_{nab}^{1}(M,C)\ar[r]^-i&\mathrm{Aut}_{C}(E)\ar[r]^-K&\mathrm{Aut}(C)\times \mathrm{Aut}(M)\ar[r]^-W&\mathrm{H}^2_{nab}(M,C)}.\]
\end{Theo}

\section{Particular case: abelian extensions of $\lambda$-weighted Rota-Baxter Lie coalgebras}
In this section, we investigate the results of previous section in particular case.
Let $(C,\Delta_C,R_C)$ and $(M,\Delta_M,R_M)$ be two $\lambda$-weighted Rota-Baxter Lie coalgebras. Let
\[\xymatrix@C=20pt{\mathcal{E}:0\ar[r]&C\ar[r]^-f&E\ar[r]^g&M\ar[r]&0.}\] be an
abelian extension of $(C,\Delta_C,R_C)$ by $(M,\Delta_M,R_M)$.
 Denote the set of all equivalent classes of abelian extensions $(C,\Delta_C,R_C)$ by $(M,\Delta_M,R_M)$ by $\mathrm{Ext}_{ab}(M,C)$.

\begin{Prop} The triple $(M,\rho,R_M)$ is a right Lie comodule of $(C,\Delta_C,R_C)$, where $\rho$ is given by (\ref{2co2}).
\end{Prop}

\begin{proof}
Since $g$ is surjective, for all $m\in M$, there is an element $e\in E$ such that
$g(e)=m$. By direct computations,
\begin{align*}
	&(I\otimes\Delta_C)\rho(m)-(\rho\otimes I)\rho(m)+(I\otimes\tau)(\rho\otimes I)\rho(m)\\
	=&(I\otimes\Delta_C)\rho g(e)-(\rho\otimes I)\rho g(e)+(I\otimes\tau)(\rho\otimes I)\rho g(e)\\
	=&(g\otimes\Delta_C t)\Delta_E(e)-(\rho g\otimes t)\Delta_E(e)+(I\otimes\tau)(\rho g\otimes t)\Delta_E(e)\\
	=&(g\otimes\Delta_C t)\Delta_E(e)-((g\otimes t)\Delta_E\otimes t)\Delta_E(e)+(I\otimes\tau)((g\otimes t)\Delta_E\otimes t)\Delta_E(e)\\
	=&(g\otimes\Delta_C t)\Delta_E(e)-(g\otimes (t\otimes t)\Delta_E)\Delta_E(e)\\
	=&(g\otimes(\Delta_C t-(t\otimes t)\Delta_E))\Delta_E(e)\\
	=&(g\otimes-h g)\Delta_E(e)=(I\otimes-h)\Delta_M(m)=0.
\end{align*}
By the same token, (\ref{rR}) holds. Thus, $(M,\rho,R_M)$ is a right Lie comodule of $(C,\Delta_C,R_C)$.
\end{proof}

\begin{rem} In the case of non-abelian extensions, $(M,\rho,R_M)$ is not a right Lie comodule of $(C,\Delta_C,R_C)$.\end{rem}

\begin{Theo} \begin{enumerate}[label=$(\roman*)$]
		\item The triple $(C\oplus
		M,\Delta_{(h,\rho)},R_{\phi})$ is a $\lambda$-weighted Rota-Baxter Lie coalgebra
		if and only if $(h,\phi)$ is a 2-cocycle of $(C,\Delta_C,R_C)$ with coefficients in the
		Lie comodule $(M,\rho,R_M)$.
		\item Abelian extensions of a $\lambda$-weighted Rota-Baxter Lie coalgebra
		$(C,\Delta_C,R_C)$ by $(M,\Delta_M,R_M)$ are classified
		by the second cohomology group $\mathrm{\bar{H}}_{RB}^{2}(M, C)$ of
		$(C,\Delta_C,R_C)$ with coefficients in $(M,\rho,R_M)$.
	\end{enumerate}
\end{Theo}

\begin{proof}
	\begin{enumerate}[label=$(\roman*)$]
\item It is similar to the proof of Proposition \ref{Pro: LCN}.\item By the same token as in Theorem \ref{Ccy2}, we can get the statement.
\end{enumerate}
\end{proof}

\begin{Theo}\label{ABEX} Let $\mathcal{E}: 0\rightarrow C\stackrel{f}{\rightarrow} E\stackrel{g}{\rightarrow}M\rightarrow 0$ be an abelian extension of a  $\lambda$-weighted Rota-Baxter Lie coalgebra
 $(C,\Delta_C,R_C)$ by $(M,\Delta_M,R_M)$ and $t$ be its retraction. Assume that $(h,\phi)$ is a 2-cocycle and $(M,\rho,R_M)$ is a right Lie comodule of $(C,\Delta_C,R_C)$ associated to $\mathcal{E}$. A pair $(\alpha,\beta)\in \mathrm{Aut}(C
)\times \mathrm{Aut}(M)$ is extensible with respect to the abelian extension $\mathcal{E}$ if and only if there is a
linear map $\varphi:M\rightarrow C$
satisfying the following conditions:
\begin{equation}\nonumber h\beta-(\alpha\otimes\alpha)h=(\varphi\otimes\alpha)\rho-\tau(\varphi\otimes\alpha)\rho
-\Delta_C\varphi,\end{equation}
 \begin{equation}\nonumber\rho\beta=(\beta\otimes\alpha)\rho,\end{equation}
\begin{equation}\nonumber\phi\beta-\alpha\phi=\varphi R_M-R_C\varphi.\end{equation}
\end{Theo}

\begin{proof}
It can be get directly from Theorem \ref{MT}.
\end{proof}

Let $\mathcal{E}: 0\rightarrow C\stackrel{f}{\rightarrow} E\stackrel{g}{\rightarrow}M\rightarrow 0$ be an abelian extension of
the$\lambda$-weighted Rota-Baxter Lie coalgebra
 $(C,\Delta_C,R_C)$ by $(M,\Delta_M,R_M)$ and $t$ be its retraction.
 Assume that $(h,\phi)$ is a 2-cocycle and $(M,\rho,R_M)$ is a right Lie comodule of $(C,\Delta_C,R_C)$ associated to $\mathcal{E}$.

The space
 \begin{equation}\nonumber C_{\rho}=\{(\alpha,\beta)\in \mathrm{Aut}(C
)\times \mathrm{Aut}(M)|\rho\beta=(\beta\otimes\alpha)\rho,~~\forall~m\in V\}\end{equation}
 is called the space of compatible pairs of automorphisms.
It is easy to verify that $C_{\rho}$ is a subgroup of $\mathrm{Aut}(C
)\times \mathrm{Aut}(M)$. For all $(\alpha,\beta)\in \mathrm{Aut}(C
)\times \mathrm{Aut}(M)$, $(h_{(\alpha,\beta)},\phi_{(\alpha,\beta)})$ may not be a 2-cocycle.
 Indeed, we have

\begin{Prop} \label{Cac1} The tuple $(h_{(\alpha,\beta)},\phi_{(\alpha,\beta)})$ is a
 2-cocycle corresponding to the abelian extension $\mathcal{E}$ if $(\alpha,\beta)\in C_{\rho}$.
\end{Prop}

\begin{proof}
Take the same procedure as the proof of Proposition \ref{EEnc0}.
\end{proof}

\begin{Theo} \label{Cac2} Let $\mathcal{E}: 0\rightarrow C\stackrel{f}{\rightarrow} E\stackrel{g}{\rightarrow}M\rightarrow 0$ be an abelian extension of a $\lambda$-weighted Rota-Baxter Lie coalgebra
 $(C,\Delta_C,R_C)$ by $(M,\Delta_M,R_M)$ and $(h,\phi)$ be a
 2-cocycle associated to $\mathcal{E}$. A pair $(\alpha,\beta)\in \mathrm{Aut}(C
)\times \mathrm{Aut}(M)$ is extensible with respect to the abelian extension
$\mathcal{E}$ if and only if the following conditions hold:
\begin{enumerate}[label=$(\roman*)$]
	\item $(\alpha,\beta)\in C_{\rho}$.
	\item $(h,\phi)$ and $(h_{(\alpha,\beta)},\phi_{(\alpha,\beta)})$ are in the same cohomological class.
\end{enumerate}
\end{Theo}

\begin{proof}
Combining Theorem \ref{EEnc} and Proposition \ref{Cac1}, we get the statement.
\end{proof}

In the case of abelian extensions, $\mathrm{Z}^{1}_{nab}(M,C)$ defined by (\ref{W5}) turns to $\mathrm{\bar{Z}}_{RB}^{1}(M,C)$ given by (\ref{Ccy1}).
 In the light of Theorem \ref{Wm5} and Theorem \ref{Cac2}, we have the following exact sequence:

\begin{Theo} There is an exact sequence:
\[\xymatrix@C=20pt{0\ar[r]&\mathrm{\bar{Z}}_{RB}^{1}(M,C)\ar[r]^-i&\mathrm{Aut}_{C}(E)\ar[r]^-K&C_{\rho}\ar[r]^-W&\mathrm{\bar{H}}^2_{RB}(M,C).}\]
\end{Theo}

\begin{center}{\textbf{Acknowledgments}}
\end{center}
Project supported by the National Natural Science Foundation of
China (11871421), the Natural Science Foundation of Zhejiang
Province of China (LY19A010001) and the Science and Technology
Planning Project of Zhejiang Province (2022C01118).

\begin{center} {\textbf{Statements and Declarations}}
\end{center}
 All datasets underlying the conclusions of the paper are available
to readers. No conflict of interest exits in the submission of this
manuscript.

	
\end{document}